\documentclass[a4paper, DIV=12, 11pt]{scrartcl}
\usepackage[latin1]{inputenc}
\usepackage{amsmath}
\usepackage{amsfonts}
\usepackage{amssymb}
\usepackage{amsthm}
\usepackage{graphicx}
\usepackage{thmtools}
\usepackage{bbm}
\usepackage{hyperref}
\usepackage{enumitem}
\usepackage{todonotes}

\usepackage{tikz}
\usetikzlibrary{matrix}

\declaretheoremstyle[headfont=\normalfont]{normalhead}
\newtheorem{lemma}{Lemma}[section]
\newtheorem{theorem}[lemma]{Theorem}
\newtheorem{proposition}[lemma]{Proposition}
\newtheorem{corollary}[lemma]{Corollary}
\newtheorem{definition}[lemma]{Definition}
\newtheorem{remark}[lemma]{Remark}

\newcounter{mt}

\newtheorem{maintheorem}[mt]{Theorem}

\newtheorem{maincorollary}[mt]{Corollary}

\newcommand{\R}{\mathbb{R}}

\newcommand{\Val}{\mathrm{Val}}
\newcommand{\vol}{\mathrm{vol}}

\newcommand{\supp}{\mathrm{supp\ }}

\newcommand{\diam}{\mathrm{diam}}

\newcommand{\sign}{\mathrm{sign}}
\newcommand{\SO}{\mathrm{SO}}
\newcommand{\nc}{\mathrm{nc}}
\newcommand{\GL}{\mathrm{GL}}
\newcommand{\VConv}{\mathrm{VConv}}
\newcommand{\Conv}{\mathrm{Conv}}
\newcommand{\Hess}{\mathrm{Hess}}
\newcommand{\D}{\mathbb{D}}

\setkomafont{title}{\normalfont\Large}

\author{Jonas Knoerr}
\title{Zonal valuations on convex bodies}
\date{}

\newcommand{\Addresses}{{
		\bigskip
		\footnotesize
		
		Jonas Knoerr, \textsc{Institute of Discrete Mathematics and Geometry, TU Wien, Wiedner Hauptstrasse 8-10, 1040 Wien, Austria}\par\nopagebreak
		\textit{E-mail address}: \texttt{jonas.knoerr@tuwien.ac.at}
		
		\medskip
}}

\makeatletter
\def\blfootnote{\xdef\@thefnmark{}\@footnotetext}
\makeatother

\makeindex
\begin{document}
\maketitle
\begin{abstract}
	A complete classification of all zonal, continuous, and translation invariant valuations on convex bodies is established. The valuations obtained are expressed as principal value integrals with respect to the area measures. The convergence of these  principal value integrals is obtained from a new weighted version of an inequality for the volume of spherical caps due to Firey. For Minkowski valuations, this implies a refinement of the convolution representation by Schuster and Wannerer in terms of singular integrals. As a further application, a new proof of the classification of $\SO(n)$-invariant, continuous, and dually epi-translation invariant valuations on the space of finite convex functions by Colesanti, Ludwig, and Mussnig is obtained.
\end{abstract}

\section{Introduction}
Let $\mathcal{K}(\R^n)$ denote the space of convex bodies in $\R^n$, that is, the set of all non-empty, compact, and convex subsets of $\R^n$ equipped with the Hausdorff metric. A map $\mu:\mathcal{K}(\R^n)\rightarrow (\mathcal{A},+)$ with values in an Abelian semi-group is called a valuation if
\begin{align*}
	\mu(K)+\mu(L)=\mu(K\cup L)+\mu(K\cap L)
\end{align*}
for all $K,L\in\mathcal{K}(\R^n)$ such that $K\cup L$ is convex. This property is shared by many geometric functionals, for example Euler characteristic, surface area, and volume, and consequently, valuations play a key role in various areas of mathematics, including convex and discrete geometry. In particular, much of the tremendous progress in integral geometry  relies on classification results for various scalar or tensor-valued valuations (see \cite{AleskerDescriptioncontinuousisometry1999,AleskerDescriptiontranslationinvariant2001,AleskerEtAlHarmonicanalysistranslation2011,BernigFuHermitianintegralgeometry2011,BernigHugKinematicformulastensor2018,FuStructureunitaryvaluation2006,LudwigReitznerclassification$SLn$invariant2010,Wannerermoduleunitarilyinvariant2014} for some of these results).

A different line of research in valuation theory is focused on \emph{Minkowski valuations}, that is, valuations with values in the space $\mathcal{K}(\R^n)$ equipped with Minkowski addition, a notion introduced by Ludwig \cite{LudwigProjectionbodiesvaluations2002,LudwigMinkowskivaluations2005}. 
In her seminal work, Ludwig showed that a key operator in convex geometry, the projection body map, is uniquely characterized by the valuation property, continuity, translation invariance, and its equivariance properties with respect to the special linear group (see \cite{AbardiaBernigProjectionbodiescomplex2011,HaberlMinkowskivaluationsintertwining2012,KiderlenBlaschkeMinkowskiendomorphisms2006,LudwigEllipsoidsmatrixvalued2003,SchneiderEquivariantendomorphismsspace1974,SchusterWannerer$GLn$contravariantMinkowski2012} for related results). 
This notion goes back to Minkowski and has since been an important tool in the study of projections of convex bodies (see for example \cite{GardnerGeometrictomography2006,SchneiderConvexbodiesBrunn2014} for an overview of some of these results). In particular, projection bodies and their polar bodies feature prominently in a variety of inequalities (see, for example, \cite{LutwakInequalitiesmixedprojection1993,PettyIsoperimetricproblems1971,ZhangRestrictedchordprojection1991}, and \cite{HaberlSchusterGeneral$L_p$affine2009,LutwakEtAl$L_p$affineisoperimetric2000,WangaffineSobolevZhang2012,ZhangaffineSobolevinequality1999} for further generalizations).\\

More recently, it was observed that several inequalities for projection bodies extend to much larger classes of translation invariant Minkowski valuations intertwining the group $\SO(n)$ (see \cite{BergEtAlLogconcavityproperties2018,HaberlSchusterAffinevs.Euclidean2019,ParapatitsSchusterSteinerformulaMinkowski2012,SchusterCroftonmeasuresMinkowski2010}). Most of these results rely on suitable representations of these valuations in terms of convolution operators and the regularity properties of the associated integral kernels (see \cite{BraunerOrtegaMorenoFixedPointsMean2023,OrtegaMorenoIterationsMinkowskivaluations2023,OrtegaMorenoSchusterFixedpointsMinkowski2021}).\\
A key representation result was established by Schuster and Wannerer in \cite{SchusterWannererMinkowskivaluationsgeneralized2018}. Let $\mathcal{M}_o(S^{n-1})$ and $C_o(S^{n-1})$ denote the spaces of signed Radon measures and continuous functions on $S^{n-1}$ that are \emph{centered}, that is, that have their center of mass at the origin.  Here and in the following, we identify $\SO(n-1)$ with the subgroup of $\SO(n)$ that fixes the axis spanned by the vector $e_n$.
\begin{theorem}[Schuster--Wannerer \cite{SchusterWannererMinkowskivaluationsgeneralized2018} Theorem 2]
	\label{theorem:Schuster_Wannerer_Minkowski}	
		If $\Phi:\mathcal{K}(\R^n)\rightarrow\mathcal{K}(\R^n)$ is a  continuous Minkowski valuation  which is translation invariant and $\SO(n)$-equivariant, then there exist uniquely determined $c_0,c_n\ge 0$, $\SO(n-1)$-invariant $\nu_j\in \mathcal{M}_o(S^{n-1})$ and $f_{n-1}\in C_o(S^{n-1})$ such that
		\begin{align*}
			h_{\Phi(K)}=c_0+\sum_{j=1}^{n-2}S_j(K)*\nu_j+S_{n-1}(K)*f_{n-1}+c_nV_n(K)
		\end{align*}
		for every $K\in\mathcal{K}(\R^n)$.
\end{theorem}
Here, $h_K$ denotes the support function of a convex body $K\in\mathcal{K}(\R^n)$, $S_j(K)$ denotes Alexandrov's area measures (compare \cite{SchneiderConvexbodiesBrunn2014}) and $*$ denotes the spherical convolution of measures, see \cite[Section 2]{SchusterWannererMinkowskivaluationsgeneralized2018} for details. Note that the result provides a precise representation formula for these Minkowski valuations, but does not establish a characterization of these functionals, since the right hand side of this equation does in general not define a support function for arbitrary functions and measures.\\

Let us remark that this representation can be refined for homogeneous valuations. Dorrek \cite{DorrekMinkowskiendomorphisms2017} showed that in this case the measures $\nu_j$ are absolutely continuous with respect to the spherical Lebesgue measure with density in $L^1(S^{n-1})$. More recently, Brauner and Ortega-Moreno established that the functions are actually locally Lipschitz continuous outside of the poles \cite{BraunerOrtegaMorenoFixedPointsMean2023}. However, these results do not apply to general Minkowski valuations, since the different homogeneous components do in general not correspond to homogeneous Minkowski valuations, compare \cite[Theorem 4.4]{DorrekMinkowskiendomorphisms2017}.\\

The proof of Theorem \ref{theorem:SchusterWannererClassification} given in \cite{SchusterWannererMinkowskivaluationsgeneralized2018} relies on a regularity result for the space $\Val(\R^n)^{\SO(n-1)}$ of all real-valued, continuous, translation invariant, and $\SO(n-1)$-invariant valuations, considered as a subspace of the space of generalized spherical valuations. We refer to \cite[Theorem 3]{SchusterWannererMinkowskivaluationsgeneralized2018} for the result as well as the necessary background on generalized valuations (see also \cite{AleskerFaifmanConvexvaluationsinvariant2014,BernigFaifmanGeneralizedtranslationinvariant2016}). Let us, however, remark, that the proof only provides a (non-sharp) regularity result, not a full description of these valuations.\\

The main result of this article provides a complete characterization of $\Val(\R^n)^{\SO(n-1)}$. Since $\SO(n-1)$-invariant functions on the sphere are often called \emph{zonal} functions, we will call valuations of this type zonal as well. Before we state our results, we require some background about the space $\Val(\R^n)$ of all continuous and translation invariant valuations $\mu:\mathcal{K}(\R^n)\rightarrow\R$. As shown by McMullen \cite{McMullenValuationsEulertype1977}, this space admits a homogeneous decomposition, 
\begin{align}
	\label{eq:McMullenDecomp}
	\Val(\R^n)=\bigoplus_{j=0}^n\Val_j(\R^n),
\end{align}
where $\mu\in\Val_j(\R^n)$ if and only if it is $j$-homogenous, that is, $\mu(tK)=t^j\mu(K)$ for all $t\ge 0$ and $K\in\mathcal{K}(\R^n)$. Furthermore, we have the following complete characterization of some of the homogeneous components:
\begin{itemize}
	\item $\Val_0(\R^n)$ is spanned by the Euler characteristic (which is equal to $1$ for all $K\in\mathcal{K}(\R^n))$.
	\item $\Val_n(\R^n)$ is spanned by the Lebesgue measure, as shown by Hadwiger \cite{HadwigerVorlesungenuberInhalt1957}.
	\item For every $\mu\in \Val_{n-1}(\R^n)$ there exists a unique $f\in C_o(S^{n-1})$ such that 
	\begin{align*}
		\mu(K)=\int_{S^{n-1}}f(v)dS_{n-1}(K,v),
	\end{align*}
	as shown by McMullen \cite{McMullenContinuoustranslationinvariant1980}. Conversely, any such function defines a continuous valuation by the same formula. Note that $f$ has to be a zonal function if $\mu$ is zonal.
\end{itemize}
Let us remark that in the remaining cases $1\le j\le n-2$ no full characterization of these valuations is known. However, due to results by Alesker \cite{AleskerDescriptiontranslationinvariant2001}, there exist powerful tools to describe various dense subspaces, which form the foundation for our approach.\\

The homogeneous decomposition and the three results listed above reduce the characterization problem to the description of the remaining components  $\Val_j(\R^n)^{\SO(n-1)}$ for $1\le j\le n-2$. Let us introduce the following classes of functions for $1\le j\le n-2$:
\begin{align*}
	D^{\frac{n-j-1}{2}}:=\left\{\right. f\in C((-1,1)):&\lim\limits_{s\rightarrow\pm 1}(1-s^2)^{\frac{n-j-1}{2}}f(s)=0,\\
	&\left. \text{the limits}~\lim_{s\rightarrow\pm1}\int_0^sf(s)(1-s^2)^{\frac{n-j-1}{2}-1}ds~\text{exist and are finite}\right\}.
\end{align*}
For $f\in D^{\frac{n-j-1}{2}}$, the function $v\mapsto f(v_n)$ is in general not integrable with respect to $S_j(K)$ for $K\in\mathcal{K}(\R^n)$. Nevertheless, these functions induce continuous valuations in terms of the following principal value integral.
\begin{maintheorem}
	\label{maintheorem:SingularValuations}
	Let $1\le j\le n-2$. For every $f\in D^{\frac{n-j-1}{2}}$ and every $K\in\mathcal{K}(\R^n)$ the limit
	\begin{align*}
		\Phi_j(f)[K]=\lim\limits_{\epsilon\rightarrow0}\int_{\{v\in S^{n-1}:|v_n|\le 1-\epsilon\}} f(v_n)dS_j(K,v)
	\end{align*}
	exists, and this defines a continuous valuation $\Phi_j(f)\in \Val_j(\R^n)^{\SO(n-1)}$.
\end{maintheorem}
We provide a precise characterization of the functions in $D^{\frac{n-j-1}{2}}$ for which the principal value is a proper integral for every convex body $K\in\mathcal{K}(\R^n)$, which applies in particular to Berg's functions, compare Corollary \ref{corollary:IntegrabilityBergsFunctions}.\\

As suggested by the notation, our approach is centered around the properties of the map 
\begin{align*}
	\Phi_j:D^{\frac{n-j-1}{2}}\rightarrow \Val_j(\R^n)^{\SO(n-1)}.
\end{align*}
More precisely, the proof of Theorem \ref{maintheorem:SingularValuations} is based on extending this map from $C([-1,1])$ to $D^{\frac{n-j-1}{2}}$ by continuity with respect to the norm 
\begin{align*}
	\|f\|_{D^{\frac{n-j-1}{2}}}:=\sup_{s\in (-1,1)}(1-s^2)^{\frac{n-j-1}{2}}|f(s)|+\sup_{s\in (-1,1)}\left|\int_{0}^s(1-t^2)^{\frac{n-j-1}{2}-1}f(t)dt\right|
\end{align*}
and the natural topology on $\Val(\R^n)$, compare Section\ref{section:Preliminaries}. 

Our main result shows that the valuations constructed above provide a complete characterization of the space $\Val_j(\R^n)^{\SO(n-1)}$ for $1\le j\le n-2$.
\begin{maintheorem}
	\label{maintheorem:Classification}
	Let $1\le j\le n-2$. For every $\mu\in \Val_j(\R^n)^{\SO(n-1)}$ there exists $f\in D^{\frac{n-j-1}{2}}$ such that $\mu=\Phi_j(f)$. This function is unique up to the addition of a linear function.
\end{maintheorem}
We provide two proofs of this result. Both rely on the fact that the function $f$ can be reconstructed by evaluating a given valuation in a family of cones with $e_n$ as axis of revolution. The first proof in Section \ref{section:ClassificationProof1} shows that the (right) inverse to this evaluation map is continuous with respect to the relevant topologies, which reduces the proof to an approximation result based on a description of smooth spherical valuations due to Schuster and Wannerer \cite{SchusterWannererMinkowskivaluationsgeneralized2018}. The second proof is based on an idea by Wannerer and relies on an implication of the Hard Lefschetz Theorem for spherical valuations (compare Section \ref{section:Prelim_valuationsBodies}), which reduces the proof to the $1$-homogeneous cases, where additional tools are available (see Section \ref{section:ClassificationProof2}).\\

Note that Theorem \ref{maintheorem:Classification} is equivalent to the fact that the map $\Phi_j:D^{\frac{n-j-1}{2}}\rightarrow\Val_j(\R^n)^{\SO(n-1)}$ is onto. Our construction implies the following stronger version of this result.

\begin{maintheorem}
	\label{maintheorem:topologicalIsomorphism}
	Let $D^{\frac{n-j-1}{2}}_o$ denote the quotient of $D^{\frac{n-j-1}{2}}$ by the $1$-dimensional subspace generated by linear functions. $\Phi_j$ descends to an isomorphism
	\begin{align*}
		D^{\frac{n-j-1}{2}}_o\rightarrow \Val_j(\R^n)^{\SO(n-1)}
	\end{align*}
	of topological vector spaces. 
\end{maintheorem}
	Note that any complement to the space of linear functions in $D^{\frac{n-j-1}{2}}$ is isomorphic to the quotient via the quotient map. In particular, we may replace $D^{\frac{n-j-1}{2}}_o$ by any such complement in the statement of Theorem \ref{maintheorem:topologicalIsomorphism}. One canonical choice is the space
	\begin{align*}
		\left\{f\in D^{\frac{n-j-1}{2}}:\int_{-1}^1f(t)t(1-t^2)^{\frac{n-3}{2}}dt=0\right\}
	\end{align*}
	of all functions whose zonal extensions to $S^{n-1}$ are centered.
	
Finally, note that the results imply the following refinement of Theorem \ref{theorem:Schuster_Wannerer_Minkowski}. 

	\begin{maincorollary}
			If $\Phi:\mathcal{K}(\R^n)\rightarrow\mathcal{K}(\R^n)$ is a  continuous Minkowski valuation  which is translation invariant and $\SO(n)$-equivariant, then there exist uniquely determined $c_0,c_n\ge 0$, and $f_j\in D^{\frac{n-j-1}{2}}$, $1\le j\le n-2$, $f_{n-1}\in C([-1,1])$ such that for all $y\in S^{n-1}$
			\begin{align*}
				h_{\Phi(K)}(y)=&c_0+\sum_{j=1}^{n-2}\lim\limits_{\epsilon\rightarrow0}\int_{\{v\in S^{n-1}: |\langle y,v\rangle|\le 1-\epsilon\}} f_j(\langle y,v\rangle)dS_j(K,v)\\
				&+\int_{S^{n-1}} f_{n-1}(\langle y,v\rangle)dS_{n-1}(K,v)+c_nV_n(K)
			\end{align*}
			for every $K\in\mathcal{K}(\R^n)$. The functions $f_1,\dots,f_{n-1}$ are unique up to linear functions.
	\end{maincorollary}

	Let us remark that the classification of $\Val_j(\R^n)^{\SO(n-1)}$ presented above closely resembles the description of a certain space of $\SO(n)$-invariant valuations on convex functions obtained by Colesanti, Ludwig, and Mussnig in terms of functional versions of the intrinsic volumes \cite{ColesantiEtAlHadwigertheoremconvex2020}. This resemblance is not accidental, since these two settings are closely related, compare \cite{KnoerrSmoothvaluationsconvex2024,KnoerrUlivellivaluationsconvexbodies2024}. In particular, our approach is heavily based on ideas by Colesanti, Ludwig, and Mussnig from \cite{ColesantiEtAlHadwigertheoremconvex,ColesantiEtAlHadwigertheoremconvex2020,ColesantiEtAlHadwigertheoremconvex2022,ColesantiEtAlHadwigertheoremconvex2023}, including the construction of the relevant valuations in Theorem \ref{maintheorem:SingularValuations} (which relies on an inequality obtained in \cite{KnoerrSingularvaluationsHadwiger2022} using the geometric ideas from \cite{ColesantiEtAlHadwigertheoremconvex2020}). This also applies to the use of the cones in the reconstruction procedure, which are the geometric analog of certain functions used in \cite{ColesantiEtAlHadwigertheoremconvex,ColesantiEtAlHadwigertheoremconvex2020,ColesantiEtAlHadwigertheoremconvex2022,ColesantiEtAlHadwigertheoremconvex2023}.  In Section \ref{section:FunctionalHadwiger} we show that the results of this article can be used to recover the classification result by Colesanti, Ludwig, and Mussnig, which is based on an alternative approach to their result discussed in \cite[Section 8.3]{ColesantiEtAlHadwigertheoremconvex2022}, and which highlights the connection between the two settings.

	\paragraph*{Plan of the article}
	Section \ref{section:Preliminaries} collects some background on convex bodies and valuations on convex bodies, in particular on the area measures of cones. We also establish some necessary results about the spaces $D^{\frac{n-j-1}{2}}$ and introduce the relevant integral transforms needed in the proofs of the characterization result. Section \ref{section:SingularIntegrals} establishes bounds on the integral of zonal functions, which we use to obtain a new proof of a theorem by Firey on the volume of spherical caps \cite{FireyLocalbehaviourarea1970}. This estimate is then used to prove Theorem \ref{maintheorem:SingularValuations}. We also describe the subset of function that lead to a proper integral representation for all convex bodies in Proposition \ref{proposition:GeneralIntegrability} and show that this includes Berg's functions. Section \ref{section:Classification} contains the two proofs of the classification result in Theorem \ref{maintheorem:Classification}. In Section \ref{section:FunctionalHadwiger} we apply these results to rotation invariant valuations on convex functions.

	\paragraph*{Acknowledgments}
	This project originates in discussions with Thomas Wannerer during a visit at the Friedrich-Schiller-Universität Jena. The second proof using the Hard Lefschetz Theorem for spherical valuations is based on his ideas, and the author thanks him for his comments and the productive discussions during the preparation of this article. \\
	The author was partially supported by DFG-grant WA 3510/3-1.

\section{Preliminaries}
	\label{section:Preliminaries}
	We denote by $\mathcal{H}^{j}$ the $j$-dimensional Hausdorff measure of a Borel subset in $\R^n$. For $j=n$, we also use $\vol_n=\mathcal{H}^n$ for the $n$-dimension Lebesgue measure. $\omega_n$ denotes the volume of the  $n$-dimension unit ball $B_1(0)$ in $\R^n$. In particular, $\mathcal{H}^{n-1}(S^{n-1})=n\omega_n$. For sets $A_i\subset\R^n$, we denote by $\mathrm{conv}(A_i)_{i\in I}$ the convex hull of their union.
	\subsection{Convex bodies and area measures of cones}
	\label{section:Prelim_valuationsBodies}
	We refer to the monograph by Schneider \cite{SchneiderConvexbodiesBrunn2014} for a comprehensive background on convex bodies and area measures and only collect the results needed in this article. First, a convex body $K\in\mathcal{K}(\R^n)$ is uniquely determined by its support function $h_K:\R^n\rightarrow\R$ defined by
	\begin{align*}
		h_K(y)=\sup_{x\in K}\langle y,x\rangle,\quad y\in \R^n.
	\end{align*}
	The support function is a convex function, so it is in particular continuous. 	Since it is also $1$-homogeneous, it is advantageous to identify it with its restriction to the unit sphere in $\R^n$, which we denote by the same symbol for brevity. In particular, we denote by $\|h_K\|_\infty$ its supremum norm as a function in $C(S^{n-1})$. Then the Hausdorff metric $d_H$ admits the following description in terms of the support function:
	\begin{align*}
		d_H(K,L)=\|h_K-h_L\|_\infty \quad\text{for}~K,L\in\mathcal{K}(\R^n).
	\end{align*}
		The support function has the following well known properties:
	\begin{enumerate}
		\item $h_{gK}(y)=h_K(g^Ty)$ for $g\in\GL(n,\R)$, $y\in \R^n$, $K\in\mathcal{K}(\R^n)$;
		\item $h_{K+L}=h_K+h_L$ for $K,L\in\mathcal{K}(\R^n)$;
		\item $h_{tK}=th_K$ for $t\ge 0$, $K\in\mathcal{K}(\R^n)$;
		\item $\max(h_{K}(y),h_L(y))=h_{\mathrm{conv}(K,L)}(y)$ for $K,L\in\mathcal{K}(\R^n)$, $y\in\R^n$.
	\end{enumerate}
	Given a convex body $K$, the $j$th area measure $S_j(K)$ is a non-negative centered measure on $S^{n-1}$. They may be obtained in the following way, compare \cite[Section 4]{SchneiderConvexbodiesBrunn2014}: For $\rho>0$ consider the set $K_\rho:=\{x\in \R^n:d(x,K)\le \rho\}$ and consider the nearest point projection $p(K,\cdot):K_\rho\setminus K\rightarrow K$. For a Borel set $\omega\subset S^{n-1}$, we define the local parallel set
	\begin{align*}
		B_\rho(K,\omega):=\left\{x\in K_\rho\setminus K: \frac{x-p(K,x)}{|x-p(K,x)|}\in\omega\right\}.
	\end{align*}
	Then  the following local Steiner formula holds, compare \cite[Theorem 4.2.1]{SchneiderConvexbodiesBrunn2014}:
	\begin{align}
		\label{equation:SteinerAreaMeasures}
		\vol_n(B_\rho(K,\omega))=\frac{1}{n}\sum_{j=0}^{n-1}\rho^{n-j}\binom{n}{j}S_j(K,\omega).
	\end{align}
	Moreover, the area measures are locally determined in the following sense. If $K,L\in\mathcal{K}(\R^n)$ are two convex bodies such that $h_K=h_L$ on an open set $U\subset S^{n-1}$, then $S_j(K)|_U=S_j(L)|_U$.\\

	As an example, consider the $(n-1)$-dimensional disk $\D^{n-1}=\{(x,0)\in \R^n: |x|\le 1,~x\in\R^{n-1}\}$. Then for $0\le j\le n-2$
	\begin{align}
		\label{eq:areaMeasureDisk}
		S_j(\D^{n-1},U)=\frac{n-1-j}{n-1}\int_{U} (1-v_n^2)^{-\frac{j}{2}}d\mathcal{H}^{n-1}(v),
	\end{align}
	compare \cite[Example 4.6]{BraunerOrtegaMorenoFixedPointsMean2023}. We will be interested in the area measures of the cones $C_h$ given for $h\in \R$ by
	\begin{align*}
		C_h=\mathrm{conv}(\D^{n-1},\{he_n\}) \in \mathcal{K}(\R^n).
	\end{align*}

	\begin{lemma}
	\label{lemma:areaMeasuresCones}
	Let $h\geq 0$, $U\subset S^{n-1}$ a Borel set.
	\begin{enumerate}
		\item For $U\subset \{ v \in S^{n-1} :  (1+h^2)^{-\frac{1}{2}}<v_n\le 1\}$ and $1\leq j\leq n-1$: $S_{j}(C_h , U) =0$.
		\item For $ U\subset \{ v \in S^{n-1} :  -1\le v_n < (1+h^2)^{-\frac{1}{2} }\}$
		and $0\leq j\leq n-2$: 
		\begin{align*}
			S_{j}(C_h , U) =  \frac{(n-j-1)}{n-1}\int_{U}  (1-v_n^2 )^{-\frac{j}{2}}  d\mathcal{H}^{n-1}(v).
		\end{align*} 
		\item $ S_{j}(C_h , \{v\in S^{n-1}:v_n=(1+h^2)^{-\frac{1}{2}}\})=\omega_{n-1} \sqrt{1+h^2}\left(\frac{h}{\sqrt{1+h^2}}\right)^{n-j-1}$.
	\end{enumerate}
\end{lemma}
\begin{proof}
	Notice that for $v\in S^{n-1}$,
	\begin{align*}
		h_{C_h}(v)=&h_{\mathrm{conv}(\{he_n\},\D^{n-1})}(v)=\max(h_{\{h e_n\}}(v),h_{\D^{n-1}}(v))\\
		=&\begin{cases}
			h_{\{he_n\}}(v), & (1+h^2)^{-\frac{1}{2}}\le v_n\le 1,\\
			h_{\D^{n-1}}(v), &  -1\le v_n\le (1+h^2)^{-\frac{1}{2}}.
		\end{cases}
	\end{align*}
	Since $S_j$ is locally determined and the sets $\{ v \in S^{n-1} :  (1+h^2)^{-\frac{1}{2}}<v_n\le 1\}$ and $\{ v \in S^{n-1} :  -1\le v_n<(1+h^2)^{-\frac{1}{2}}\}$ are open, this implies 
	\begin{align*}
		S_j(C_h,U)=\begin{cases}
			S_j(\{he_n\},U), & U\subset \{ v \in S^{n-1} :  (1+h^2)^{-\frac{1}{2}}<v_n\le 1\},\\
			S_j(\D^{n-1},U), & U\subset \{ v \in S^{n-1} :  -1\le v_n<(1+h^2)^{-\frac{1}{2}}\}.
		\end{cases}
	\end{align*}
	This shows 1. and 2. (compare \eqref{eq:areaMeasureDisk}).\\
	Let us turn to $3$. Set $\omega=\{v\in S^{n-1}:v_n=(1+h^2)^{-\frac{1}{2}}\}$. Note that
	\begin{align*}
		B_\rho(C_h,\omega)=\left\{\left((1-\frac{x_n}{h})v,x_n\right)+s(1+h^2)^{-\frac{1}{2}}\left(hv,1\right):x_n\in[0,h],~s\in (0,\rho],~ v\in S^{n-2}\right\}.
	\end{align*}
	This is a solid of revolution and an elementary calculation shows that its volume is given by
	\begin{align*}
		\vol_n(B_\rho(C_h,\omega))=\frac{1}{n}\frac{1+h^2}{h}\left[\left(1+\rho\frac{h}{\sqrt{1+h^2}}\right)^{n}-\rho^{n}\left(\frac{h}{\sqrt{1+h^2}}\right)^{n}-1\right].
	\end{align*}
	The claim follows from the Steiner formula for area measures \eqref{equation:SteinerAreaMeasures}.
\end{proof}

	\begin{lemma}
		\label{lemma:IntegrationCones}
		Writing $s= \sign(h)(1+h^2)^{-\frac 12 }$, we have for $f\in C([-1,1])$ and all $h\in\R$
		\begin{align*}
			\int_{S^{n-1}} &f(v_n) dS_j(C_{h},v) \\   =&   \sign(s)\omega_{n-1} \left(  {s}^{-1} (1-s^2)^{\frac{n-j-1}{2}}  f(s)  + (n-j-1) \int_{-\sign(s)1}^s f(t) (1-t^2)^{\frac{n-j-1}{2}-1} dt\right).
		\end{align*}
	\end{lemma}
	\begin{proof}
		Since the expressions are symmetric under a sign change for $h$, it is sufficient to establish the claim for $h\ge0$ (or equivalently $s>0$).
		From Lemma \ref{lemma:areaMeasuresCones}, we obtain for $f\in C([-1,1])$ and $h>0$,
		\begin{align*}
			&\int_{S^{n-1}}f(v_n)dS_j(C_h,v)\\
			=&\int_{\left\{v_n=\frac{1}{\sqrt{1+h^2}}\right\}}f(v_n)dS_j(C_h,v)+\frac{(n - j -1)}{n-1}\int_{\left\{-1\le v_n\le \frac{1}{\sqrt{1+h^2}}\right\}}f(v_n)(1-v_n^2)^{-j}d\mathcal{H}^{n-1}(v)\\
			=&\omega_{n-1} \sqrt{1+h^2}\left(\frac{h}{\sqrt{1+h^2}}\right)^{n-j-1}f\left(\frac{h}{\sqrt{1+h^2}}\right)+(n-j-1)\omega_{n-1}\int_{-1}^{\frac{1}{\sqrt{1+h^2}}}f(t)(1-t^2)^{\frac{n-3-j}{2}}dt,
		\end{align*}
		where we switched to spherical cylinder coordinates in the last step. The claim follows.
	\end{proof}

\subsection{The spaces $D^{a}$}
\label{section:Da}
\begin{definition}
	\label{def:Dj}
	For $a>0$ let $D^{a}$ be the space of all continuous functions $ f\colon (-1,1)\to \R $
	satisfying the properties 
	\begin{enumerate}
		\item $\lim_{s\to \pm 1} (1-s^2)^{a} f(s) =0$,
		\item the limits $\lim_{s\rightarrow\pm 1} \int_0^s (1-t^2)^{a-1} f(t) dt $ exist and are finite.
	\end{enumerate}
\end{definition}
We equip $D^{a}$ with the norm
\begin{align*}
	\|f\|_{D^{a}}:=\sup_{s\in (-1,1)}(1-s^2)^{a}|f(s)|+\sup_{s\in (-1,1)}\left|\int_{0}^s(1-t^2)^{a-1}f(t)dt\right|.
\end{align*}
\begin{lemma}
	\label{lemma:DaBanach}
	$D^{a}$ is complete with respect to $\|\cdot\|_{D^{a}}$
\end{lemma}
\begin{proof}
	Let $(f_m)_m$ be a Cauchy sequence in $D^{a}$. Then $\phi_m(s):=(1-s^2)^{a}f_m(s)$ extends by assumption to a continuous function on $[-1,1]$ with $\phi_m(\pm1)=0$. Moreover, $(\phi_m)_m$ is a Cauchy sequence in $C([-1,1])$, as
	\begin{align*}
		\|\phi_m-\phi_k\|_\infty\le \|f_m-f_k\|_{D^{a}},
	\end{align*}
	and thus converges uniformly on $[-1,1]$ to some $\phi\in C([-1,1])$. In particular, $\phi(\pm1)=0$. Set $f(s):=(1-s^2)^{-a}\phi(s)$ for $s\in (-1,1)$. Then $f\in C((-1,1))$ and
	\begin{align*}
		\lim\limits_{s\rightarrow\pm 1}(1-s^2)^{a}f(s)=\lim\limits_{s\rightarrow\pm1}\phi(s)=0.
	\end{align*}
	Moreover,
	\begin{align*}
		\sup_{s\in(-1,1)} |(1-s^2)^{a}f(s)-(1-s^2)^{a}f_m(s)|=\|\phi-\phi_m\|_\infty
	\end{align*}
	converges to $0$, so $(f_m)_m$ converges locally uniformly to $f$ on $(-1,1)$. In particular, for any $s\in(-1,1)$:
	\begin{align*}
		&\left|\int_0^s(1-t^2)^{a}f(t)dt-\int_0^s(1-t^2)^{a}f_m(t)dt\right|\\
		=&\lim\limits_{k\rightarrow \infty}	\left|\int_0^s(1-t^2)^{a}f_k(t)dt-\int_0^s(1-t^2)^{a}f_m(t)dt\right|\\
		\le& \liminf_{k\rightarrow\infty}\|f_k-f_m\|_{D^{a}}.
	\end{align*}
	As $(f_m)_m$ is a Cauchy sequence, this implies that the functions $\psi_m(s):=\int_0^s(1-t^2)^{a}f_m(t)dt$, which extend to elements of $C([-1,1])$ by assumption, converge uniformly on $[-1,1]$ to a function $\psi\in C([-1,1])$, which satisfies
	\begin{align*}
		\psi(s)=\int_0^s(1-t^2)^{a}f(t)dt \quad\text{for}~s\in (-1,1).
	\end{align*}
	In particular, the limits
	\begin{align*}
		\lim\limits_{s\rightarrow\pm 1}\int_0^s(1-t^2)^{a}f(t)dt=\psi(\pm 1)
	\end{align*}
	exist and are finite. Thus $f\in D^{a}$. It easily follows from the definition of $\|\cdot\|_{D^{a}}$ that
	\begin{align*}
		\|f-f_m\|_{D^{a}}\le\liminf_{k\rightarrow\infty}\|f_k-f_m\|_{D^{a}},
	\end{align*}
	so $(f_m)_m$ converges to $f$ with respect to this norm as $(f_m)_m$ is a Cauchy sequence.
\end{proof}
\begin{lemma}
	\label{lemma:ContFctDenseDj}
	For $f\in D^{a}$ and $r\in (0,1)$ define $f^r\in C([-1,1])$ by
	\begin{align*}
		f^r(s)=\begin{cases}
			f(r), & s\ge r,\\
			f(s), & -r<s<r,\\
			f(-r), & s\le r.
		\end{cases}
	\end{align*} 
	Then $f^r$ converges to $f$ in $D^{a}$ for $r\rightarrow1$.
\end{lemma}
\begin{proof}
	For $f\in D^{a}$ consider the function $\psi(s)=\int_0^s(1-t^2)^{a-1}f(t)dt$. Then $\psi$ extends to a continuous function on $[-1,1]$. As $f$ and $f^r$ coincide on $[-r,r]$, we have
	\begin{align*}
		\|f-f^r\|_{D^{a}}=&\sup_{|s|\ge r}(1-s^2)^{a}|f(s)-f(r)|\\
		&+\sup_{s\in [r,1)}\left|\int_r^s(1-t^2)^{a}f(t)dt-\int_r^s(1-t^2)^{a}f(r)dt\right|\\
		&+\sup_{s\in (-1,-r])}\left|\int_s^{-r}(1-t^2)^{a}f(t)dt-\int_s^{-r}(1-t^2)^{a}f(-r)dt\right|\\
		\le&2\sup_{|s|\ge r}(1-s^2)^{a}|f(s)|+\sup_{s\in [r,1]}|\psi(s)-\psi(r)|+\sup_{s\in[r,1)}\frac{|f(r)|}{r}\int_r^s(1-s^2)^{a-1}sds\\
		&+\sup_{s\in [-1,-r]}|\psi(s)-\psi(-r)|+\sup_{s\in(-1,-r]}\frac{|f(-r)|}{r}\int_s^{-r}(1-s^2)^{a-1}(-s)ds.
	\end{align*}
	As $\psi$ is continuous in $\pm 1$ and $f\in D^{a}$, we obtain
	\begin{align*}
		&\limsup_{r\rightarrow1}\|f-f^r\|_{D^{a}}\\
		\le&\limsup_{r\rightarrow1}\frac{|f(r)|}{r}\int_r^1(1-s^2)^{a-1}sds+\frac{|f(-r)|}{r}\int_{-1}^{-r}(1-s^2)^{a-1}(-s)ds\\
		=&\frac{1}{a}\limsup_{r\rightarrow1}\frac{(|f(r)|+|f(-r)|)(1-r^2)^{a}}{r}=0.
	\end{align*}
	Thus $f^r$ converges to $f$ in $D^{a}$ for $r\rightarrow1$.
\end{proof}

\subsection{An integral transform}
The integral transforms in this section are motivated by Lemma \ref{lemma:IntegrationCones}. Consider for $f\in D^a$ the function $I_a(f):[-1,1]\setminus\{0\}\rightarrow\R$ defined by
\begin{align*}
	I_a(f)[s]=&\sign(s)\left[s^{-1} (1-s^2)^a f(s) + 2a \int_{-\sign(s)}^s f(t) (1-t^2)^{a-1} dt\right],\quad s\ne 0.
\end{align*}
Note that we interpret the integral as an improper Riemann integral, which is well defined since $f\in D^{a}$.

\begin{definition}
	Let $\mathcal{C}$ denote the space of all continuous functions $u:[-1,1]\setminus\{0\}\rightarrow\R$ such that
	\begin{enumerate}
		\item $u(1)=u(-1)$,
		\item the limit $\lim\limits_{s\rightarrow0}|s|u(s)$ exists and is finite.
	\end{enumerate}
\end{definition}

We equip $\mathcal{C}$ with the norm
\begin{align*}
	\|u\|_{\mathcal{C}}:=\sup_{s\in [-1,1]\setminus\{0\}}|s|\cdot|u(s)|.
\end{align*}

\begin{corollary}
	\label{corollary:IaProperties}
	$I_a:D^a\rightarrow\mathcal{C}$ is well defined and continuous. Moreover, $I_a(f)$ has the following properties:
	\begin{enumerate}
		\item $I_a(f)[1]=I_a(f)[-1]=\int_{-1}^1f(t)(1-t^2)^{a-1}dt$.
		\item If $f\in D^a$ is even or odd, then $I_a(f)$ is even or odd, respectively.
	\end{enumerate}
\end{corollary}
\begin{proof}
	It is easy to see that $I_a(f)\in\mathcal{C}$. Note that
	\begin{align*}
		|s| I_a(f)[s]=
			(1-s^2)^a f(s) + 2a s\int_{-\sign(s)}^s f(t) (1-t^2)^{a-1} dt,\quad s\ne0.
	\end{align*}
	In particular, $\|I_a(f)\|_{\mathcal{C}}\le (1+4a)\|f\|_{D^a}$. Thus $I_a$ is continuous. The rest of the properties follow directly from the definition
\end{proof}

Note that the definition of $I_a(f)$ can equivalently be stated as 
\begin{align}
	\label{eq:FormulaIaDerivative}
	I_af(s)=&|s|^{-1}(1-s^2)^{a+1}\frac{d}{ds}\left[(1-s^2)^{-a}\int_{-\sign(s)}^s f(t) (1-t^2)^{a-1} dt\right],
\end{align}
so if we set $u=I_a(f)$, then for $s\ne 0$
\begin{align*}
	&\sign(s)\int_{0}^s t(1-t^2)^{-(a+1)}u(t)dt\\
	=&(1-s^2)^{-a}\int_{-\sign(s)}^s f(t) (1-t^2)^{a-1} dt-\int_{-\sign(s)}^0 f(t) (1-t^2)^{a-1} dt.
	\end{align*}
We may multiply this equation with $(1-s^2)^a$, take the derivative, and rearrange, to obtain for $s\ne 0$
\begin{align*}
	f(s)+2as\int_{-\sign(s)}^0f(t)(1-t^2)^{a-1}dt=&-2a|s|\int_{0}^s t(1-t^2)^{-(a+1)}u(t)dt+|s|(1-s^2)^{-a}u(s).
\end{align*}
If $f$ is an odd function, then the left hand side differs from $f$ by a linear function, since
\begin{align*}
	\int_{-\sign(s)}^0f(t)(1-t^2)^{a-1}dt
\end{align*}
is independent of $\sign(s)$ in this case. If $f$ is even, then the left hand side differs from $f$ by the function
\begin{align*}
	2a|s|\cdot \frac{1}{2}\int_{-1}^1f(t)(1-t^2)^{a-1}dt=\frac{|s|}{2}I_a(f)[1]=\frac{|s|}{2}u(1).
\end{align*}

\begin{definition}
	\label{definition:AbelTypeTransform}
	For $u\in \mathcal{C}$, we define $J_a(u):(-1,1)\rightarrow\R$ by
	\begin{align*}
		J_a(u)[s]=\begin{cases}
			|s|\frac{-u(1)}{2}+|s|(1-s^2)^{-a}u(s)-2a |s|\int_{0}^st(1-t^2)^{-(a+1)}u(t)dt, & s\ne0,\\
			\lim_{t\rightarrow0}|t|u(t),& s=0.
		\end{cases}
	\end{align*}
\end{definition}
The following follows directly from the previous discussion.
\begin{lemma}
	\label{lemma:CompositionJaIa}
	If $f\in D^a$ is even, then $J_a\circ I_a(f)=f$. If $f\in D^a$ is odd, then
	\begin{align*}
		J_a\circ I_a(f)[s]=f(s)-s\cdot 2a\int_{0}^1f(t)(1-t^2)^{a-1}dt,\quad s\in (-1,1).
	\end{align*}
\end{lemma}

\begin{proposition}
	\label{proposition:ContinutiyJa}
	The map $J_a:\mathcal{C}\rightarrow D^a$ is well defined and continuous. Moreover, $I_a\circ J_a=Id$.
\end{proposition}
\begin{proof}
	We start with showing that $J_a(u)\in D^a$ for $u\in \mathcal{C}$. We first establish that
	\begin{align}
		\label{eq:limitIntegral_u}
		\lim\limits_{s\rightarrow\pm1}\sign(s)2a (1-s^2)^{a}\int_{0}^st(1-t^2)^{-(a+1)}u(t)dt=u(\pm 1).
	\end{align}
	We will only consider the limit $s\rightarrow1$, the other case is similar.	With the change of coordinates $x=\frac{1-t}{1-s}$ we obtain
	\begin{align*}
		&	2a (1-s^2)^{a}\int_{0}^st(1-t^2)^{-(a+1)}u(t)dt\\
		=&2a(1+s)^a(1-s)^a\int_{0}^s(1-t)^{-(a+1)}(1+t)^{-(a+1)}tu(t)dt\\
		=&2a (1+s)^{a} \int_{1}^{\frac{1}{1-s}}x^{-(a+1)}(2-(1-s)x)^{-(a+1)}\left[u(1-(1-s)x)(1-(1-s)x)\right]dx.
	\end{align*}
	As $a>0$, the function $x\mapsto x^{-(a+1)}$ is integrable over $[1,\infty)$. Moreover, $x\mapsto u(1-(1-s)x)(1-(1-s)x)$ is uniformly bounded by $\|u\|_{\mathcal{C}}$ for $x\in [1,\frac{1}{1-s}]$, and this function converges pointwise to $u(1)$ for $s\rightarrow1$. As $(2-(1-s)x)^{-(a+1)}\le 1$ for $x\in [1,\frac{1}{1-s}]$, dominated convergence thus implies
	\begin{align*}
			&\lim\limits_{s\rightarrow1}2a (1-s^2)^a\int_{0}^st(1-t^2)^{-(a+1)}u(t)dt=u(1)\cdot a \int_{1}^{\infty}x^{-(a+1)}dx=u(1).
	\end{align*} 
	We thus obtain \eqref{eq:limitIntegral_u}, and it is easy to see that this implies $\lim\limits_{s\rightarrow\pm 1}(1-s^2)^aJ_a(u)[s]=0$.\\
	Moreover, the change of variables also provides the estimate
	\begin{align}
		\label{eq:estimateJa1}
		&\left|2a (1-s^2)^{a}\int_{0}^st(1-t^2)^{-(a+1)}u(t)dt\right|
		\le a 2^{a+1} \int_{1}^{\infty}x^{-(a+1)}dx\|u\|_{\mathcal{C}}.
	\end{align}
	Let us now turn to the second condition in Definition \ref{def:Dj}. Note that for $u\in \mathcal{C}$, $s\ne 0$,
	\begin{align}
		\label{eq:FormulaJaDerivative}
		J_a(u)[s]=
		|s|\frac{-u(1)}{2}+\sign(s)(1-s^2)^{-(a-1)}\frac{d}{ds}\left[(1-s^2)^{a}\int_{0}^s t(1-t^2)^{-(a+1)}u(t)dt\right].
	\end{align}
	In particular, 
	\begin{align}
		\label{eq:estimateJaHelpEquation}
			&\sign(s)\int_0^sJ_a(u)[t](1-t^2)^{a-1}dt\\
			\notag
			=&\frac{-u(1)}{2}\int_0^st(1-t^2)^{a-1}dt+\int_0^s\frac{d}{dx}\left[(1-x^2)^{a}\int_{0}^x t(1-t^2)^{-(a+1)}u(t)dt\right]dx\\
			\notag
			=&\frac{-u(1)}{2}\int_0^st(1-t^2)^{a-1}dt+(1-s^2)^{a}\int_{0}^s t(1-t^2)^{-(a+1)}u(t)dt.
	\end{align}
	Up to factor of $2a$, the second integral is the same as in \eqref{eq:limitIntegral_u}, and we have shown that this term converges to $u(1)$. Thus
	\begin{align*}
		\lim\limits_{s\rightarrow1}\int_0^sJ_a(u)[t](1-t^2)^{a-1}dt=\frac{-u(1)}{2}\int_0^1t(1-t^2)^{a-1}dt+\frac{u(1)}{2a},
	\end{align*}
	which is finite as $a>0$. The same holds when we consider the limit $s\rightarrow-1$, so $J_a(u)\in D^a$.\\
	In addition, applying the estimate in \eqref{eq:estimateJa1} to \eqref{eq:estimateJaHelpEquation}, we obtain
	\begin{align}
		\label{eq:estimateJa2}
		\left|\int_0^sJ_a(u)[t](1-t^2)^{a-1}dt\right|\le \frac{|u(1)|}{2}\int_0^st(1-t^2)^{a-1}dt+2^{a}\int_1^\infty x^{-(a+1)}dx \|u\|_{\mathcal{C}},
	\end{align}
	where both terms are bounded by constant multiple of $\|u\|_{\mathcal{C}}$ independent of $u\in\mathcal{C}$. Combining \eqref{eq:estimateJa1} and \eqref{eq:estimateJa2}, we obtain a constant $C>0$ such that
	\begin{align*}
		\|J_a(u)\|_{D^a}\le C\|u\|_{\mathcal{C}}.
	\end{align*}
	Thus $J_a:\mathcal{C}\rightarrow D^a$ is well defined and continuous.\\
	Finally, let us show that $I_a\circ J_a(u)=u$. Combining \eqref{eq:FormulaIaDerivative} and \eqref{eq:FormulaJaDerivative}, we obtain for $u\in \mathcal{C}$
	\begin{align*}
		I_a\circ J_a(u)[s]=&|s|^{-1}(1-s^2)^{a+1}\frac{d}{ds}\left[(1-s^2)^{-a}\frac{-u(1)}{2}\int_{-\sign(s)}^s |t| (1-t^2)^{a-1} dt \right]\\
		&+s^{-1}(1-s^2)^{a+1}\frac{d}{ds}\left[(1-s^2)^{-a}\int_{-\sign(s)}^s\frac{d}{dt}\left[(1-t^2)^{a}\int_{0}^t x(1-x^2)^{-(a+1)}u(x)dx\right] dt\right].
	\end{align*}
	For the first term, we calculate
	\begin{align*}
		\int_{-\sign(s)}^s |t| (1-t^2)^{a-1} dt=&\sign(s)\left[-\frac{1}{2a}(1-s^2)^a+\frac{1}{a}\right].
	\end{align*}
	Applying \eqref{eq:limitIntegral_u} to the second term, we thus obtain for $s\ne0$,
	\begin{align*}
		I_a\circ J_a(u)[s]=&s^{-1}(1-s^2)^{a+1}\frac{d}{ds}\left[(1-s^2)^{-a}\frac{-u(1)}{2a} \right]\\
		&+s^{-1}(1-s^2)^{a+1}\frac{d}{ds}\left[(1-s^2)^{-a}\left[(1-s^2)^{a}\int_{0}^s x(1-x^2)^{-(a+1)}u(x)dx-\frac{-u(-1)}{2a}\right] \right]\\
		=&u(s).
	\end{align*}
	Here we used that $u(1)=u(-1)$ in the last step. This completes the proof.
\end{proof}

\subsection{Valuations on convex bodies}
	
For a comprehensive background on valuations on convex bodies we refer to \cite{AleskerIntroductiontheoryvaluations2018,BernigAlgebraicintegralgeometry2011,FuAlgebraicintegralgeometry2014} as well as \cite[Section 6]{SchneiderConvexbodiesBrunn2014}. In this section we collect the results needed in this article.\\

We start with the following geometric construction of valuations on convex bodies (and refer to \cite{AleskerTheoryvaluationsmanifolds.2006} for further details). Note that to any $K\in\mathcal{K}(\R^n)$, we can associated the set
\begin{align*}
	\nc(K):=\{(x,v)\in\R^n\times S^{n-1}:v\text{ outer normal to } K \text{ in }x\in\partial K\},
\end{align*}
which is a compact Lipschitz submanifold of the sphere bundle $S\R^n:=\R^n\times S^{n-1}$ of dimension $n-1$ that carries a natural orientation. We may thus consider $\nc(K)$ as an integral $(n-1)$-current on $S\R^n$, called the \emph{normal cycle} of $K$, which we will denote by $\omega\mapsto \nc(K)[\omega]$, where $\omega\in\Omega^{n-1}(S\R^n)$ is a smooth $(n-1)$-form. More generally, we can associate to any $\omega\in \Omega^{n-1}(S\R^n)$ and $K\in\mathcal{K}(\R^n)$ the signed Radon measure on $S^{n-1}$ given by
\begin{align*}
	\Phi_\omega(K)[U]:=\nc(K)[1_{\pi_2^{-1}(U)}\omega]\quad\text{for}~U\subset S^{n-1}~\text{Borel set},
\end{align*}
where $\pi_2:\R^n\times S^{n-1}\rightarrow S^{n-1}$ denotes the projection onto the second factor. If the differential form $\omega$ is translation invariant, then the map $K\mapsto \Phi_\omega$ is a smooth area measure in the sense of \cite{WannererIntegralgeometryunitary2014,Wannerermoduleunitarilyinvariant2014}. For example, choose standard coordinates $(w_1,\dots,w_n)$ on $\R^n$ with induced coordinates $(v_1,\dots,v_n)$ on $S^{n-1}$. Consider the differential forms on $\R^n\times S^{n-1}$ for $0\le j\le n-1$ given by
\begin{align}
	\label{eq:definitonKappaNJ}
	\kappa^n_j:=\frac{1}{j!(n-j-1)!}\sum_{\pi\in S_n}\sign(\pi)v_{\pi(1)}dw_{\pi(2)}\dots dw_{\pi(j+1)}dv_{\pi(j+2)}\dots dv_{\pi(n)},
\end{align}
where we suppress the wedge product in the notation. Then 
\begin{align}
	\label{equation:DifferentialFormSurfaceArea}
	S_j(K)=\frac{1}{(n-j)\omega_{n-j}}\Phi_{\kappa^n_j}(K),
\end{align}
compare \cite[Section 2.1]{FuAlgebraicintegralgeometry2014}. In particular, integrals of continuous functions with respect to the area measures can be interpreted as integrals over the normal cycle. Let us remark that the differential forms are not uniquely determined by the associated measure-valued functionals or valuations. We refer to the work by Bernig and Bröcker \cite{BernigBroeckerValuationsmanifoldsRumin2007} for a description of the kernel and only remark that the normal cycle vanishes on multiples of the contact form $\alpha\in \Omega^1(S\R^n)$ given by
\begin{align*}
	\alpha=\sum_{i=1}^nw_idv_i.
\end{align*}

Let us now return to the space $\Val(\R^n)$ of all continuous and translation invariant valuations on $\mathcal{K}(\R^n)$. First, the homogeneous decomposition \eqref{eq:McMullenDecomp}
implies that 
\begin{align*}
	\|\mu\|=\sup_{K\subset B_1(0)}|\mu(K)|\quad\text{for}~\mu\in \Val(\R^n)
\end{align*}
defines a norm that metrizes the topology of uniform convergence on compact subsets on $\Val(\R^n)$, so in particular, this equips $\Val(\R^n)$ with the structure of a Banach space. We have a natural continuous representation of the general linear group $\GL(n,\R)$ on $\Val(\R^n)$ given by
\begin{align*}
	[\pi(g)\mu](K)=\mu(g^{-1}K)\quad\text{for}~g\in\GL(n,\R),~\mu\in\Val(\R^n),~K\in\mathcal{K}(\R^n).
\end{align*}
The smooth vectors of this representation are called \emph{smooth valuations}, that is, a valuation $\mu\in\Val(\R^n)$ is called smooth if the map
\begin{align*}
	\GL(n,\R)&\rightarrow\Val(\R^n)\\
	g&\mapsto \pi(g)\mu
\end{align*}
is a smooth map. This notion was first introduced by Alesker \cite{AleskerHardLefschetztheorem2003}. We denote the subspace of smooth valuation by $\Val(\R^n)^{\infty}$ and remark that it is a Fr\'echet space with respect to the G\r{a}rding topology. Note that this notion is obviously compatible with the homogeneous decomposition, so $\Val_j(\R^n)^{\infty}= \bigoplus_{j=0}^n\Val_j(\R^n)^{\infty}$. Furthermore, $\Val_j(\R^n)^{\infty}\subset \Val_j(\R^n)$ is dense with respect to the norm topology. By averaging any approximating sequence with respect to the Haar measure on $\SO(n-1)$, this implies that $\Val_j(\R^n)^{\infty,\SO(n-1)}$ is dense in $\Val_j(\R^n)^{\SO(n-1)}$.\\

Recall that we are considering elements in $\Val_j(\R^n)^{\SO(n-1)}$. These belong to the subspace $\Val_j(\R^n)^{\mathrm{sph}}$ of so-called \emph{spherical valuations} considered in \cite{SchusterWannererMinkowskivaluationsgeneralized2018}. We omit the precise definition and only remark that the space $\Val_j(\R^n)^{\infty,\mathrm{sph}}$ of smooth spherical valuations is dense in $\Val_j(\R^n)^{\mathrm{sph}}$ and contains the space $\Val_j(\R^n)^{\infty,\SO(n-1)}$, compare \cite [Section 3]{SchusterWannererMinkowskivaluationsgeneralized2018}. In addition, both spaces are invariant under the group $\SO(n)$, and we have the following characterization of smooth spherical valuations.
\begin{theorem}[Schuster--Wannerer~\cite{SchusterWannererMinkowskivaluationsgeneralized2018} Theorem 4.1]
	\label{theorem:SchusterWannererClassification}
	For $1\le j\le n-1$, the map $E_j:C^\infty_o(S^{n-1})\rightarrow\Val_j(\R^n)^{\infty,\mathrm{sph}}$, defined by
	\begin{align*}
		[E_j(f)](K)=\int_{S^{n-1}}fdS_j(K),
	\end{align*}
	is an $\SO(n)$-equivariant isomorphism of topological vector spaces.
\end{theorem}

Let $\Lambda:\Val_j(\R^n)\rightarrow\Val_{j-1}(\R^n)$ denote the Lefschetz operator given by
\begin{align*}
	\Lambda\mu (K)=\frac{d}{dt}\Big|_0\mu(K+B_1(0)).
\end{align*}
We refer to \cite{AleskerHardLefschetztheorem2003,BernigBroeckerValuationsmanifoldsRumin2007} for the Hard-Lefschetz Theorem for smooth valuations and only discuss the properties of this operator relevant to this article. Note that $\Lambda$ is well defined and continuous with respect to the norm topology due to McMullen's decomposition \eqref{eq:McMullenDecomp}, since $t\mapsto \mu(K+tB_1(0))$ is a polynomial of degree at most $n$ in $t\ge 0$. Moreover, it is obviously $\SO(n)$-equivariant. It follows from the local Steiner formula \eqref{equation:SteinerAreaMeasures} that $\Lambda E_j=jE_{j-1}$. This directly implies the following (which was implicitly established in the proof of \cite[Theorem 4.1]{SchusterWannererMinkowskivaluationsgeneralized2018}).
\begin{corollary}
	\label{corollary:LefschetzIsomorphismSpherical}
	For $2\le j\le n-1$, $\Lambda:\Val_j(\R^n)^{\infty,\mathrm{sph}}\rightarrow\Val_{j-1}(\R^n)^{\infty,\mathrm{sph}}$ is an $\SO(n)$-equivariant topological isomorphism.
\end{corollary}
The following consequence will be used in the second proof of Theorem \ref{maintheorem:Classification} in Section \ref{section:ClassificationProof2}
\begin{proposition}
	\label{proposition:LefschetzInjectiveZonal}
	For $2\le j\le n-1$, the map $\Lambda:\Val_j(\R^n)^{\mathrm{sph}}\rightarrow\Val_{j-1}(\R^n)^{\mathrm{sph}}$ is injective.
\end{proposition}

\begin{remark}
	Note that Theorem \ref{maintheorem:Classification} implies in particular that the restriction of $\Lambda$ to $\SO(n-1)$-invariant valuations is not surjective for any $2\le j\le n-1$ since $D^{\frac{n-j-1}{2}}$ is a proper subspace of $D^{\frac{n-(j-1)-1}{2}}$.
\end{remark}
\begin{proof}[Proof of Proposition \ref{proposition:LefschetzInjectiveZonal}]
	Take a smooth approximation $\phi_m\in C_c^\infty(\SO(n))$ of $\delta_{Id}$ and consider for $\mu\in \Val_j(\R^n)^{\mathrm{sph}}$ the valuation
	\begin{align*}
		\mu_m(K):=\int_{\SO(n)}\phi_m(g)\mu(g^{-1}K)dg.
	\end{align*}
	It follows from the discussion in \cite[Section 3]{SchusterWannererMinkowskivaluationsgeneralized2018} that $\mu_m$ is a smooth spherical valuation for every $m\in\mathbb{N}$. In addition, the sequence $(\mu_m)_m$ converges to $\mu$ in $\Val(\R^n)$. Moreover, since $\Lambda$ is $\SO(n)$-equivariant and continuous with respect to the norm topology on $\Val_j(\R^n)^{\mathrm{sph}}$, we obtain
	\begin{align*}
		\Lambda\mu_m(K)=\int_{\SO(n)}\phi_m(g)\Lambda\mu(g^{-1}K)dg.
	\end{align*}
	If $\Lambda\mu=0$, this implies $\Lambda\mu_m=0$ for all $m\in\mathbb{N}$. Since $\Lambda$ is injective on smooth spherical valuations by Corollary \ref{corollary:LefschetzIsomorphismSpherical}, $\mu_m=0$ for all $m$ and consequently $\mu=0$.
\end{proof}

\section{Principal value integrals with respect to area measures}
	\label{section:SingularIntegrals}
	\subsection{Bounds on the integral of zonal functions}
	\label{section:SingularValuationsEstimates}
	
	Our proof of Theorem \ref{maintheorem:SingularValuations} is based on the following inequality.
	\begin{theorem}
		\label{theorem:estimateIntegrationSj}
		Let $1\le j\le n-2$. There exists a constant $C_{n,j}>0$ such that
		\begin{align*}
			\left|\int_{S^{n-1}}f(v_n)dS_j(K,v)\right|\le C_{n,j}\|h_K\|_\infty^j\|f\|_{D^{\frac{n-j-1}{2}}}
		\end{align*}
		for all $f\in C([-1,1])$, $K\in\mathcal{K}(\R^n)$.
	\end{theorem}

	Before we turn to the proof of Theorem \ref{theorem:estimateIntegrationSj}, let us discuss an application to the volume of the spherical caps 
	\begin{align*}
		C_r(v):=\{u\in S^{n-1}: \langle v,u\rangle >r\},\quad v\in S^{n-1},~r\in [0,1].
	\end{align*}
	The following result was established by Firey \cite{FireyLocalbehaviourarea1970} by approximation from a corresponding statement for polytopes. We give a new proof (with a different constant) based on Theorem \ref{theorem:estimateIntegrationSj}.
	\begin{theorem}[\cite{FireyLocalbehaviourarea1970}]
		\label{theorem:FireyVolumeSphericalCap}
		Let $1\le i\le n-2$ and $K\in\mathcal{K}(\R^n)$ a convex body. There exists a constant $A_{n,j}>0$ independent of $K$ such that for every $v\in S^{n-1}$,
		\begin{align*}
			S_j(K)[C_r(v)]\le A_{n,j}(\diam K)^{j}(1-r^2)^{\frac{n-j-1}{2}}\quad \text{for}~r\in[0,1].
		\end{align*}
	\end{theorem}
	\begin{proof}
		We may without loss of generality assume that $v=e_n$. As the surface area measures are translation invariant, we may further assume that $\diam(K)=\|h_K\|_\infty$. Let $f\in C([-1,1])$ be a function with $0\le f\le 1$ and $f\le 1_{[r,1]}$. Note that
		\begin{align*}
			\|f\|_{D^{\frac{n-j-1}{2}}}=&\sup_{s\in (-1,1)}(1-s^2)^{\frac{n-j-1}{2}}|f(s)|+\sup_{s\in (-1,1)}\left|\int_{0}^s(1-s^2)^{\frac{n-j-3}{2}}f(s)ds\right|\\
			\le&(1-r^2)^\frac{n-j-1}{2}+\int_{r}^1(1-s^2)^{\frac{n-j-3}{2}}ds\\
			\le&(1-r^2)^\frac{n-j-1}{2}+(1-r)(1-r^2)^{\frac{n-j-3}{2}}\\
			\le&(1-r^2)^\frac{n-j-1}{2}+(1-r^2)(1-r^2)^{\frac{n-j-3}{2}}=2(1-r^2)^\frac{n-j-1}{2}.
		\end{align*}
		Theorem \ref{theorem:estimateIntegrationSj} implies that
		\begin{align*}
				\int_{S^{n-1}}f(v_n)dS_j(K,v)\le C_{n,j}\|h_K\|_\infty^j\|f\|_{D^{\frac{n-j-1}{2}}}\le 2C_{n,j}\diam(K)^j(1-r^2)^\frac{n-j-1}{2}.
		\end{align*}
		We thus obtain
		\begin{align*}
			S_j(K)[C_r(e_n)]=\sup_{\substack{f\in C([-1,1])\\0\le f\le 1_{[r,1]}}}\int_{S^{n-1}}f(v_n)dS_j(K,v)\le 2C_{n,j}\diam(K)^j(1-r^2)^\frac{n-j-1}{2},
		\end{align*}
		which shows the claim with $A_{n,j}=2 C_{n,j}$.
	\end{proof}

	Let us now turn to the proof of Theorem \ref{theorem:estimateIntegrationSj}. We will split the integral into three different parts by removing a neighborhood of the equator $\{v_n=0\}$. For the sets containing the poles, we will use the following localization procedure: Note that the integrals given above can be realized by integrating the product of the given function with the differential form $\kappa_j^n$ over the normal cycle $\nc(K)$ of $K$, compare Section \ref{section:Prelim_valuationsBodies}. If we restrict ourselves to an open half sphere around one of the poles, this allows us to exploit a relation between the normal cycle of $K$ and the \emph{differential cycle} of the convex function $h_K(\cdot,-1)$ defined on $\R^{n-1}$. For simplicity, we discuss this object for functions defined on $\R^n$. Let $\Conv(\R^{n},\R)$ denote the space of all convex functions $h:\R^{n}\rightarrow\R$ and fix an orientation on $\R^{n}$. If $h$ is of class $C^2$, then the graph of its differential defines an oriented submanifold of $T^*\R^{n}$ of dimension $n$ and thus an integral current. Fu \cite{FuMongeAmperefunctions.1989} extended this construction to a large class of functions, including all convex functions, and it was shown in \cite{KnoerrSmoothvaluationsconvex2024} that this extension is continuous with respect to a suitable topology on $\Conv(\R^{n},\R)$ and the space of integral currents on $T^*\R^{n}$. We will denote the differential cycle of $h\in \Conv(\R^{n},\R)$ by $D(h)$.\\
	
	Let us return to the function $h_K(\cdot,-1)\in\Conv(\R^{n-1},\R)$ for $K\in \mathcal{K}(\R^n)$. We require the following relation between $\nc(K)$ and $D(h_K(\cdot,-1))$. Set $S^n_-:=\{v\in S^n: v_n<0\}$, $S_-\R^n:=\R^n\times S^n_-$. Consider the map
	\begin{align*}
		Q: S_-^{n-1}\R^n&\rightarrow\R^{n-1}\times\R^{n-1}\\
		(w_1,\dots,w_n,v_1,\dots,v_n)&\mapsto\left(-\frac{v_1}{v_n},\dots,-\frac{v_{n-1}}{v_n},w_1,\dots,w_{n-1}\right). 
	\end{align*}
	\begin{proposition}[\cite{KnoerrSmoothvaluationsconvex2024} Proposition 6.1]
		\label{propostion:Relation NormalDiffCycle}
		For $K\in\mathcal{K}(\R^n)$, $Q_*(\nc(K)|_{S_-\R^n})=(-1)^nD(h_K(\cdot,-1))$.
	\end{proposition}
	\begin{remark}
		The additional factor of $(-1)^n$ comes from a different choice of orientation for the normal cycle in \cite{KnoerrSmoothvaluationsconvex2024}. This choice does not play any role in the estimates below.
	\end{remark}
	Recall that we chose coordinates $(w_1,\dots,w_n)$ on $\R^n$ with induced coordinates $(v_1,\dots,v_n)$ on $S^{n-1}$ in Section \ref{section:Prelim_valuationsBodies} to define the differential from $\kappa^n_j$ on $S\R^n$ \eqref{eq:definitonKappaNJ}. On $T^*\R^n\cong\R^n\times \R^n$ we denote the corresponding coordinates by $(x,y)=(x_1,\dots,x_n,y_1,\dots,y_n)$. Consider the following differential forms on $T^*\R^n\cong\R^n\times \R^n$ for $0\le j\le n$:
	\begin{align*}
		\beta=&\sum_{i=1}^{n}x_idy_i,\quad\quad\quad \gamma=\sum_{i=1}^{n}x_idx_i,\\
		\tilde{\kappa}^n_j=&\frac{1}{j!(n-j)!}\sum_{\pi\in S_n}\sign(\pi)dx_{\pi(1)}\dots dx_{\pi(n-j)}dy_{\pi(n-j+1)}\dots dy_{\pi(n)},\\
		\tau^n_j:=&\frac{1}{j!(n-j-1)!}\sum_{\pi\in S_n}\sign(\pi)x_{\pi(1)}dx_{\pi(2)}\dots dx_{\pi(n-j)}dy_{\pi(n-j+1)}\dots dy_{\pi(n)}.
	\end{align*}
	The following result relates integrals with respect to the area measures to these forms (in dimension $n-1$).
\begin{proposition}
	Let $f\in C_c([-1,0))$. Then 
	\begin{align*}
		\int_{S_-^{n-1}}f(v_n) dS_j(K,v)
		=\frac{1}{(n-j)\omega_{n-j}}D(h_K(\cdot,-1))\left[\frac{f\left(-\frac{1}{\sqrt{1+|x|^2}}\right)}{\sqrt{1+|x|^2}^{n-j}} \left[\tilde{\kappa}^{n-1}_j+\beta\wedge \tau^{n-1}_{j-1}\right]\right].
	\end{align*}
\end{proposition}
\begin{proof}
	Recall that integrals with respect to the area measures may be expressed in terms of integrals over the normal cycle with respect to the form $\kappa^n_j$ defined in \eqref{eq:definitonKappaNJ}, compare \eqref{equation:DifferentialFormSurfaceArea}. This form satisfies
	\begin{align*}
		\kappa^n_j=&\frac{(-1)^{n-1}v_n}{j!(n-1-j)!}\sum_{\pi\in S_{n-1}}\sign(\pi)dw_{\pi(1)}\dots dw_{\pi(k)}dv_{\pi(k+1)}\dots dv_{\pi(n-1)}\\
		&+\frac{(-1)^{n-2}dw_n}{(j-1)!(n-1-j)!}\sum_{\pi\in S_{n-1}}\sign(\pi)v_{\pi(1)}dw_{\pi(2)}\dots dw_{\pi(j)}dv_{\pi(j+1)}\dots dv_{\pi(n-1)}\\
		&+\frac{(-1)^{n-2}dv_n}{j!(n-2-j)!}\sum_{\pi\in S_{n-1}}\sign(\pi)v_{\pi(1)}dw_{\pi(2)}\dots dw_{\pi(j+1)}dv_{\pi(j+2)}\dots dv_{\pi(n-1)}.
	\end{align*}
	As $\alpha=\sum_{i=1}^{n}v_idw_i$, we have $dw_n=\frac{1}{v_n}(\alpha-\sum_{i=1}^{n-1}v_idw_i)$ on $S^{n-1}_-$. Since $\nc(K)$ vanishes on multiples of $\alpha$ and $f$ is supported on $S^{n-1}_-$, the valuation $K\mapsto \int_{S^{n-1}}f(v_n) dS_j(K,v)$ is therefore given by integrating the product of $\frac{1}{(n-j)\omega_{n-j}}f(v_n)$ and the form 
	\begin{align*}
		\tilde{\omega}^n_j:=&\frac{(-1)^{n-1}v_n}{j!(n-1-j)!}\sum_{\pi\in S_{n-1}}\sign(\pi)dw_{\pi(1)}\dots dw_{\pi(j)}dv_{\pi(j+1)}\dots dv_{\pi(n-1)}\\
		&-\sum_{i=1}^{n-1}\frac{v_i}{v_n}dw_i \frac{(-1)^n}{(j-1)!(n-1-j)!}\sum_{\pi\in S_{n-1}}\sign(\pi)v_{\pi(1)}dw_{\pi(2)}\dots dw_{\pi(j)}dv_{\pi(j+1)}\dots dv_{\pi(n-1)}\\
		&+dv_n \frac{(-1)^n}{j!(n-2-j)!}\sum_{\pi\in S_{n-1}}\sign(\pi)v_{\pi(1)}dw_{\pi(2)}\dots dw_{\pi(j+1)}dv_{\pi(j+2)}\dots dv_{\pi(n-1)}
	\end{align*}
	over the normal cycle $\nc(K)$, compare \eqref{equation:DifferentialFormSurfaceArea}. Consider the map \begin{align*}
		\tilde{Q}:\R^{n-1}\times\R^{n-1}&\rightarrow\R^n\times S^{n-1}\\
		(x,y)&\mapsto \left((y,0),\frac{(x,-1)}{\sqrt{1+|x|^2}}\right).
	\end{align*}
	Then
	\begin{align*}
		\tilde{Q}^*dv_i=&\begin{cases}
			\frac{1}{\sqrt{1+|x|^2}^3}\gamma, & i=n,\\
			\frac{dx_i}{\sqrt{1+|x|^2}}-\frac{x_i}{\sqrt{1+|x|^2}^3}\gamma, & 1\le i\le n-1,
		\end{cases}
	\end{align*}
	and a short calculation shows that 
	\begin{align*}
		\tilde{Q}^*\tilde{\omega}_j
		=&\frac{(-1)^n}{\sqrt{1+|x|^2}^{n-j}}\left(\tilde{\kappa}^{n-1}_j+\beta\wedge \tau^{n-1}_{j-1}\right).
	\end{align*}
	Next, note that $\tilde{Q}\circ Q(w_1,\dots,w_n,v_1,\dots,v_n)=(w_1,\dots,w_{n-1},0,v_1,\dots,v_n)$ for $(w,v)\in S_-\R^n$. As $f\tilde{\omega}_j$ is a translation invariant form that does not contain a multiple of $dw_n$, we deduce \begin{align*}
		f\tilde{\omega}_j=Q^*\tilde{Q}^*(f\tilde{\omega}_j)=Q^*\left(\frac{(-1)^nf\left(-\frac{1}{\sqrt{1+x^2}}\right)}{\sqrt{1+|x|^2}^{n-j}}\left(\tilde{\kappa}^{n-1}_j+\beta\wedge \tau^{n-1}_{j-1}\right)\right).
	\end{align*}
	Thus
	\begin{align*}
		\int_{S^{n-1}}f(v_n) dS_j(K,v)=&\frac{1}{(n-j)\omega_{n-j}}\nc(K)[f \kappa^n_j]=\frac{1}{(n-j)\omega_{n-j}}\nc(K)[f \tilde{\omega}_j]\\
		=&\frac{1}{(n-j)\omega_{n-j}}(-1)^nQ_*D(h_K(\cdot,-1))[f \tilde{\omega}_j]\\
		=&\frac{1}{(n-j)\omega_{n-j}}D(h_K(\cdot,-1))\left[\frac{f\left(-\frac{1}{\sqrt{1+|x|^2}}\right)}{\sqrt{1+|x|^2}^{n-j}} \left[\tilde{\kappa}^{n-1}_j+\beta\wedge \tau^{n-1}_{j-1}\right]\right].
	\end{align*}
\end{proof}
\begin{remark}
	Similar relations between Monge-Amp\`ere-type operators and mixed surface area measures were obtained by Hug, Mussnig, and Ulivelli in \cite{HugEtAlAdditivekinematicformulas2024,HugEtAlKubotatypeformulas2024}.
\end{remark}

	Fix a differential form $\omega\in \Omega^{n-j}(\R^n)\otimes \Lambda^j((\R^n)^*)^*$. Since $D(f)$ is an integral current, we may define for $f\in \Conv(\R^n,\R)$ a signed measure $\Psi(f)$ on $\R^n$ by setting
	\begin{align*}
		\Psi_\omega(f)[U]=D(f)[1_{\pi^{-1}(U)}\omega]\quad U\subset\R^n~\text{bounded Borel set},
	\end{align*}
	where $\pi:T^*\R^n\rightarrow\R^n$ denotes the natural projection, compare the discussion in \cite{KnoerrMongeAmpereoperators2024}.\\
	
	The second ingredient for the proof of Theorem \ref{theorem:estimateIntegrationSj} is the following general estimate for integrals of these measures for homogeneous differential forms on $T^*\R^n$. Here, we call a differential form $\omega\in \Omega^{n-j}(\R^n)\otimes\Lambda^j((\R^n)^*)^*$ homogeneous of degree $n-j+l$ for some $l\ge0$ if $G_t^*\omega=t^{n-j+l}\omega$ for all $t\ge0$, where
	\begin{align*}
		G_t:T^*\R^n&\rightarrow T^*\R^n\\
		(x,y)&\mapsto (tx,y).
	\end{align*}
	Let $|\Psi_\omega(f)|$ denote the variation of the measure $\Psi_\omega(f)$ for $f\in\Conv(\R^n,\R)$. For $a>0$ and a locally integrable function $\zeta:(0,\infty)\rightarrow\R$, we consider the following expression:
	\begin{align*}
		\|\zeta\|_{a}:=\sup_{t>0}t^a|\zeta(t)|+\sup_{t>0} \left|\int_t^\infty\zeta(s)s^{a-1}ds\right|.
	\end{align*}
	\begin{proposition}[\cite{Knoerrgeometricdecompositionunitarily2024} Corollary~4.6]
		\label{proposition:GeneralBoundMA}
		Let $0\le j\le n-1$ and $\omega\in \Omega^{n-j}(\R^n)\otimes\Lambda^j((\R^n)^*)^*$ be homogeneous of degree $n-j+l$, $l\ge 0$. Then there exists a constant $A(\omega)$ such that for every $R>0$, for every Baire function $\phi:[0,\infty)\rightarrow[0,\infty]$ with $\supp\phi\subset[0,R]$, and every $f\in\Conv(\R^n,\R)$,
	\begin{align*}
		\int_{\R^n}\phi(|x|) d|\Psi_\omega(f)|\le A(\omega) \left(\sup_{|x|\le R+1}|f(x)|\right)^j \| \phi\|_{n-j+l}.
	\end{align*}
	\end{proposition}
	Note that Proposition \ref{proposition:GeneralBoundMA} only applies to non-negative functions. We require the following stronger estimate for $\Psi_{\tilde{\kappa}^n_j}$.
	
	\begin{proposition}[\cite{KnoerrSingularvaluationsHadwiger2022} Proposition 6.6]
		\label{prop:inequFuncIntVol}
		Let $1\le j\le n-1$. For every $\zeta\in C_c([0,\infty))$ with $\supp\zeta\subset [0,R]$ for $R>0$ and every $f\in \Conv(\R^n,\R)$
		\begin{align*}
			\left|\int_{\R^n} \zeta(|x|)d\Psi_{\tilde{\kappa}^n_j}(f)\right|\le 2^{3j+1}\omega_n\binom{n}{j}\left(\sup_{|x|\le R+1}|f(x)|\right)^j\|\zeta\|_{n-j}.
		\end{align*}
	\end{proposition}
	Let us remark that this measure is multiple of the $j$th Hessian measure, compare \cite[Section 5.1]{KnoerrSingularvaluationsHadwiger2022}.
	\begin{proof}[Proof of Theorem \ref{theorem:estimateIntegrationSj}]
		Fix $\phi\in C([-1,1])$ with $0\le \phi\le 1$, $\phi=1$ on $[-1,-\frac{3}{4}]$, $\phi=0$ on $[-\frac{1}{2},1]$. Set $\phi_+=\phi$, $\phi_-=\phi(-\cdot)$, $\phi_0=1-\phi_+-\phi_-$. Then for $f\in C([-1,1])$, $K\in\mathcal{K}(\R^n)$,
		\begin{align*}
			\int_{S^{n-1}}f(v_n) dS_j(K,v)=&\int_{S^{n-1}}\phi_+(v_n)f(v_n) dS_j(K,v)+\int_{S^{n-1}}\phi_0(v_n)f(y_n) dS_j(K,v)\\
			&+\int_{S^{n-1}}\phi_-(v_n)f(v_n) dS_j(K,v).
		\end{align*} 
		The second term can be estimated by
		\begin{align}
			\notag
			\left|\int_{S^{n-1}}\phi_0(v_n)f(v_n) dS_j(K,v)\right|\le &\sup_{|s|\le \frac{1}{4}}|f(s)|\cdot V_j(K)\\
			\notag
			\le& \sup_{|s|\le \frac{3}{4}}|(1-s^2)^{\frac{n-j-1}{2}}f(s)|\cdot \frac{1}{(1-\frac{9}{16})^{\frac{n-j-1}{2}}}\frac{\diam(K)^j}{2^jn}V_j(B_1(0))\\
			\label{eq:estimateEquatorTerm}
			\le& C_1\|f\|_{D^{\frac{n-j-1}{2}}}\diam(K)^j.
		\end{align}
		Here, $V_j(K)=nS_j(K,S^{n-1})$ denotes the $j$th intrinsic volume, and we used that $V_j$ is $j$-homogeneous and monotone.
		The first and last term can be treated similarly, so will only consider the last term. Using Proposition \ref{propostion:Relation NormalDiffCycle}, we obtain
		\begin{align*}
			(n-j)\omega_{n-j}\int_{S^{n-1}}\phi_-(v_n)f(v_n) dS_j(K,v)=D(h_K(\cdot,-1))[\tilde{f}\tilde{\kappa}^{n-1}_j]+D(h_K(\cdot,-1))[\tilde{f}\beta\wedge\tau^{n-1}_{j-1}],
		\end{align*}
		where $\tilde{f}(t)=\frac{1}{\sqrt{1+t^2}^{n-j}}\left(\phi_-\cdot f\right)\left(-\frac{1}{\sqrt{1+t^2}}\right)$. Note that $\tilde{f}$ is supported on $[0,R]$ for $R=\frac{\sqrt{7}}{3}$. Moreover, $\beta\wedge\tau^{n-1}_{j-1}$ is homogeneous of degree $n-1-j+2$.
		We apply Proposition \ref{prop:inequFuncIntVol} to the first term and Proposition \ref{proposition:GeneralBoundMA} to the second. We obtain a constant $C>0$ such that
		\begin{align}
				\label{eq:EstimateSMinus}
				\left|\int_{S^{n-1}}\phi_-(v_n)f(v_n) dS_j(K,v)\right|\le C\left(\sup_{|x|\le R+1}|h_{K}(x,-1)|\right)^j\left(\|\tilde{f}\|_{n-1-j}+\| |\tilde{f}|\|_{n-j+1}\right).
		\end{align}
		Notice that the second term involves the semi-norm $\|\cdot\|_{n-j+1}$ for the absolute value $|\tilde{f}|$. Since $|\tilde{f}|$ is supported on $[0,R]$, we obtain
		\begin{align*}
			\||\tilde{f}|\|_{n-j+1}=&\sup_{t>0}\left|\int_t^\infty |\tilde{f}(s)|s^{n-j}ds\right|+\sup_{t>0} t^{n-j+1}|\tilde{f}(t)|\\
			\le&\int_t^R sds\cdot\sup_{t>0}t^{n-j-1}|\tilde{f}(t)|+R^2\sup_{t>0}\left|t^{n-j-1}\tilde{f}(t)\right|\\
			\le& C_R\|\tilde{f}\|_{n-1-j}
		\end{align*}
		for some constant $C_R>0$ independent of $\tilde{f}$. In particular, \eqref{eq:EstimateSMinus} is bounded by a multiple of $\|\tilde{f}\|_{n-1-j}$. We will estimate
		\begin{align*}
			\|\tilde{f}\|_{n-1-j}=& \sup_{t>0}\left|\int_t^\infty \frac{[\phi_-\cdot f](-\frac{1}{\sqrt{1+s^2}})}{\sqrt{1+s^2}^{n-j}}s^{n-2-j}ds\right|+\sup_{t>0}\left|s^{n-1-j}\frac{[\phi_-\cdot f](-\frac{1}{\sqrt{1+s^2}})}{\sqrt{1+s^2}^{n-j}}\right|.
		\end{align*}
		by a constant multiple of $\|f\|_{D^{\frac{n-j-1}{2}}}$. Using the substitution $t=-\frac{1}{\sqrt{1+s^2}}$ for $s>0$ (or equivalently $s=-\frac{\sqrt{1-t^2}}{t}$), the second term can be estimated by
		\begin{align*}
			\sup_{t\in(-1,0]}\left|\sqrt{1-t^2}^{n-1-j}t\phi_-(t)f(t)\cdot\right|\le \sup_{t\in(-1,0]}\left|\sqrt{1-t^2}^{n-1-j}f(t)\right|\le \|f\|_{D^{\frac{n-j-1}{2}}},
		\end{align*}
		whereas the first term is bounded by
		\begin{align*}
		\sup_{t\in(-1,0]}\left|\int_0^t \phi_-(x)f(x)(-x)^{n-j}\left(\frac{\sqrt{1-x^2}}{x}\right)^{n-3-j}\frac{2}{x^3}dx\right|
		= &2\sup_{t\in(-1,0]}\left|\int_0^t \phi_-(x)f(x)(1-x^2)^{\frac{n-j-3}{2}}dx\right|.
		\end{align*}
		Note that $\phi_-$ is supported on $[-1,-\frac{1}{2}]$ and is equal to $1$ on $[-1,-\frac{3}{4}]$, so
		\begin{align*}
			&\sup_{t\in(-1,0]}\left|\int_0^t \phi_-(x)f(x)(1-x^2)^{\frac{n-j-3}{2}}dx\right|\\
			\le&\sup_{t\in[-\frac{3}{4},-\frac{1}{2}]}\left|\int_0^t \phi_-(x)f(x)(1-x^2)^{\frac{n-j-3}{2}}dx\right|+2\sup_{t\in(-1,-\frac{3}{4}]}\left|\int_{0}^t f(x)(1-x^2)^{\frac{n-j-3}{2}}dx\right|\\
			\le& \int_{-\frac{3}{4}}^{-\frac{1}{2}}\phi_-(s)(1-x^2)^{-1}dx\cdot\sup_{t\in(-1,1)} (1-s^2)^{\frac{n-j-1}{2}}|f(s)|+2\sup_{t\in(-1,1)}\left|\int_0^tf(x)(1-x^2)^{\frac{n-j-3}{2}}dx\right|\\
			\le& C' \|f\|_{D^{\frac{n-j-1}{2}}}
		\end{align*}
		for a constant $C'>0$ depending only on $\phi_-$. Thus $\|\tilde{f}\|_{n-1-j}$ is bounded by a constant multiple of $\|f\|_{D^{\frac{n-j-1}{2}}}$. In particular, \eqref{eq:EstimateSMinus} provides the estimate
		\begin{align*}
			\left|\int_{S^{n-1}}\phi_-(y_n)f(y_n) dS_j(K,y)\right|\le C''\left(\sup_{|x|\le R+1}|h_{K}(x,-1)|\right)^j\|f\|_{D^{\frac{n-j-1}{2}}}
		\end{align*}
		for some constant $C''>0$ independent of $f$. Next, note that
		\begin{align*}
			\sup_{|x|\le R+1} |h_K(x,-1)|=\sup_{|x|\le R+1} \sqrt{1+|x|^2}\left|h_K\left(\frac{(x,-1)}{\sqrt{1+|x|^2}}\right)\right|\le \sqrt{1+(R+1)^2}\|h_K\|_{\infty}.
		\end{align*}
		Combining these estimates with \eqref{eq:estimateEquatorTerm}, we find a constant $C_{n,j}$ such that
		\begin{align*}
			&\left|\int_{S^{n-1}}f(y_n)dS_j(K,y)\right|\le C_{n,j}\|h_K\|_\infty^j \|f\|_{D^{\frac{n-j-1}{2}}}
		\end{align*}
		for all $K\in\mathcal{K}(\R^n)$. The claim follows.	
	\end{proof}

\subsection{Principal value integrals and proof of Theorem \ref{maintheorem:SingularValuations}}
	In this section we prove Theorem \ref{maintheorem:SingularValuations}. We first show in Theorem \ref{theorem:ExtensionPhi} that the map
		\begin{align}
			\label{eq:formulaPhiContinuousFunction}
		\begin{split}
			C([-1,1])&\rightarrow\Val_j(\R^n)^{\SO(n-1)}\\
			f&\mapsto \left[K\mapsto \int_{S^{n-1}}f(u_n)dS_j(K,u)\right]
		\end{split}
	\end{align}
	extends by continuity to $D^{\frac{n-j-1}{2}}$  and then verify that these extensions admit the desired principal value representation in Theorem \ref{theorem:PrincipalValue}.
	\begin{theorem}
		\label{theorem:ExtensionPhi}
		Let $1\le j\le n-2$. The map in \eqref{eq:formulaPhiContinuousFunction} extends uniquely to a continuous map
		\begin{align*}
			\Phi_j:D^{\frac{n-j-1}{2}}\rightarrow\Val_j(\R^n)^{\SO(n-1)}.
		\end{align*}
		with 
		\begin{align*}
			\|\Phi_j(f)\|\le C_{n,j}\|f\|_{D^{\frac{n-j-1}{2}}}\quad \text{for}~f\in D^{\frac{n-j-1}{2}}.
		\end{align*}
	\end{theorem}
	\begin{proof}
		For simplicity, we denote the map in \eqref{eq:formulaPhiContinuousFunction} by $\Phi_j$ as well. 
		By Theorem \ref{theorem:estimateIntegrationSj}, there is a constant $C_{n,j}>0$ such that the estimate
		\begin{align*}
				|\Phi_j(f)[K]|=\left|\int_{S^{n-1}}f(y_n)dS_j(K,y)\right|\le C_{n,j}\|h_K\|_\infty^j\|f\|_{D^{\frac{n-j-1}{2}}}
		\end{align*}
		holds for all $f\in C([-1,1])$ and $K\in\mathcal{K}(\R^n)$. In particular, for every $f\in C([-1,1])$
		\begin{align*}
			\|\Phi_j(f)\|=\sup_{K\subset B_1(0)}|\Phi_j(f)[K]|\le C_{n,j}\|f\|_{D^{\frac{n-j-1}{2}}}.
		\end{align*}
		Thus, the map 
		\begin{align*}
			\Phi_j:C([-1,1])\rightarrow \Val_j(\R^n)^{\SO(n-1)}
		\end{align*}
		is continuous with respect to the norm $\|\cdot\|_{D^{\frac{n-j-1}{2}}}$. As $C([-1,1])$ is dense in $D^{\frac{n-j-1}{2}}$ by Lemma \ref{lemma:ContFctDenseDj} and $\Val_j(\R^n)^{\SO(n-1)}$ is a Banach space, $\Phi_j$ thus extends uniquely by continuity to a continuous map $\Phi_j:D^{\frac{n-j-1}{2}}\rightarrow \Val_j(\R^n)^{\SO(n-1)}$ with 
		\begin{align*}
			\|\Phi_j(f)\|\le C_{n,j}\|f\|_{D^{\frac{n-j-1}{2}}}\quad \text{for}~f\in D^{\frac{n-j-1}{2}}.
		\end{align*}
	\end{proof}

	The following result completes the proof of Theorem \ref{maintheorem:SingularValuations}.
		\begin{theorem}
			\label{theorem:PrincipalValue}
			For $f\in D^{\frac{n-j-1}{2}}$, $K\in  \mathcal{K}(\R^n)$,
			\begin{align*}
				\Phi_j(f)[K]=\lim\limits_{r\rightarrow1}\int_{\{v\in S^{n-1}:|v_n|\le r\}}f(v_n)dS_j(K,v).
			\end{align*}
			Moreover, the convergence is uniform on compact subsets in $\mathcal{K}(\R^n)$.
		\end{theorem}
		\begin{proof}
			For $r\in (0,1)$ let $f^r$ denote the function defined in Lemma \ref{lemma:ContFctDenseDj}. Then
			\begin{align*}
				&\left|\Phi_j(f)[K]-\int_{\{v\in S^{n-1}:|v_n|\le r\}}f(v_n)dS_j(K,v)\right|\\
				\le &|\Phi_j(f)[K]-\Phi_j(f^r)[K]|+\left|\Phi_j(f^r)[K]-\int_{\{v\in S^{n-1}:|v_n|\le r\}}f^r(v_n)dS_j(K,v)\right|.
			\end{align*}
		The first term converges uniformly to $0$ on compact subsets of $\mathcal{K}(\R^n)$ for $r\rightarrow1$ due to the fact that $\Phi_j$ is continuous and that $f^r$ converges to $f$ in $D^{\frac{n-j-1}{2}}$ by Lemma \ref{lemma:ContFctDenseDj}. For the second term, note that since $f^r$ is a continuous function, $\Phi_j(f^r)$ is given by \eqref{eq:formulaPhiContinuousFunction}, so we have
		\begin{align*}
			\left|\Phi_j(f^r)[K]-\int_{\{v\in S^{n-1}:|v_n|\le r\}}f^r(v_n)dS_j(K,v)\right|=&\left|\int_{\{v\in S^{n-1}:|v_n|> r\}}f^r(v_n)dS_j(K,v)\right|\\
			\le&\left(|f(r)|+|f(-r)|\right)\int_{\{v\in S^{n-1}:v_n> r\}}dS_j(K,v)\\
			\le&\left(|f(r)|+|f(-r)|\right)A_{n,j}\diam(K)^j(1-r^2)^{\frac{n-j-1}{2}},
		\end{align*} where we have used Firey's Theorem \ref{theorem:FireyVolumeSphericalCap} in the last step. Since $f\in D^{\frac{n-j-1}{2}}$, this term converges uniformly to $0$ on every bounded subset of $\mathcal{K}(\R^n)$. The claim follows.
 		\end{proof}
 
 		The representation in Theorem \ref{theorem:PrincipalValue} naturally leads to the question for which of functions $f$ and bodies $K\in\mathcal{K}(\R^n)$ the representation is given by a proper integral. The following result shows that there is no restriction except for integrability.
 		\begin{corollary}
 			\label{corollary:RepresentationPhiIntegrableCase}
 		Let $K\in\mathcal{K}(\R^n)$, $f\in D^{\frac{n-j-1}{2}}$. If $v\mapsto f(v_n)$ is integrable with respect to $S_j(K)$, then
 		\begin{align*}
 			\Phi_j(f)[K]=\int_{S^{n-1}}f(v_n)dS_j(K,v).
 		\end{align*}
 			This is in particular the case if $h_K$ is of class $C^2$ in a neighborhood of $\pm e_n$.
 	\end{corollary}
 	\begin{proof}
 		The first part follows directly from Theorem \ref{theorem:PrincipalValue} using dominated convergence. The second part follows from the fact that the measure $S_j(K)$ is absolutely continuous with respect to the spherical Lebesgue measure on a neighborhood of $\pm e_n$ with locally bounded density if $h_K$ is of class $C^2$ on a neighborhood of $e_n$ (see \cite[Section 4.2]{SchneiderConvexbodiesBrunn2014}), and $v\mapsto f(v_n)$ is integrable with respect to the spherical Lebesgue measure.
 	\end{proof}
 	The following result provides a characterization of those functions for which the representation is always a proper integral.
 	\begin{proposition}
 		\label{proposition:GeneralIntegrability}
 		Let $1\le j\le n-2$. For $f\in D^{\frac{n-j-1}{2}}$, the following are equivalent:
 		\begin{enumerate}
 			\item the function $v\mapsto f(v_n)$ is integrable with respect to $S_j(K)$ for every $K\in\mathcal{K}(\R^n)$;
 			\item the function $v\mapsto f(v_n)$ is integrable with respect to $S_j(\D^{n-1})$;
 			\item $\int_{-1}^1 |f(t)|(1-t^2)^{\frac{n-j-3}{2}}dt<\infty$;
 			\item $|f|\in D^{\frac{n-j-1}{2}}$.
 		\end{enumerate} 
 		In this case,
 		\begin{align*}
 			\Phi(f)[K]=\int_{S^{n-1}}f(v_n)dS_j(K,v)\quad\text{for every}~K\in\mathcal{K}(\R^n).
 		\end{align*}
 	\end{proposition}
 	\begin{proof}
 		Note that the last statement is a direct consequence of 1. and Corollary \ref{corollary:RepresentationPhiIntegrableCase}. Let us thus show that the four statements are equivalent.\\
 		The implications 1 $\Rightarrow 2$ and 3 $\Rightarrow$ 4 are trivial. Let us show 2 $\Rightarrow 3$. If the function $v\mapsto f(v_n)$ is integrable with respect to $S_j(\D^{n-1})$, the same applies to $v\mapsto |f(v_n)|$, so by approximating $|f|$ by a sequence of continuous functions and using monotone convergence, Lemma \ref{lemma:IntegrationCones} implies (noting that $\D^{n-1}=C_0$)
 		\begin{align*}
 			\int_{S^{n-1}}|f(v_n)|dS_j(\D^{n-1},v)=\omega_{n-1}(n-j-1)\int_{-1}^1|f(t)|(1-t^2)^{\frac{n-j-3}{2}}dt,
 		\end{align*}
 		so $\int_{-1}^1|f(t)|(1-t^2)^{\frac{n-j-3}{2}}dt<\infty$.\\
 		It remains to show 4 $\Rightarrow$ 1. Let us thus assume that $|f|\in D^{\frac{n-j-1}{2}}$. Since $S_j(K)$ is a non-negative measure for any $K\in\mathcal{K}(\R^n)$, we may apply monotone convergence and Theorem \ref{theorem:PrincipalValue} to obtain
 		\begin{align*}
 			\int_{S^{n-1}} |f(v_n)|dS_j(K,v)=\lim\limits_{r\rightarrow1}\int_{\{v\in S^{n-1}:|v_n|\le r\}}|f(v_n)|dS_j(K,v)=\Phi_j(|f|)[K],
 		\end{align*}
 		which is finite since $|f|\in D^{\frac{n-j-1}{2}}$. Thus $y\mapsto f(y_n)$ is integrable with respect to $S_j(K)$.
 	\end{proof}
 	Let us remark that not every $f\in D^{\frac{n-j-1}{2}}$ satisfies $|f|\in D^{\frac{n-j-1}{2}}$. An example can be constructed using the same argument as in the proof of \cite[Corollary 4.18]{Knoerrgeometricdecompositionunitarily2024}.\\

 		Let us mention the following application of this result: In his solution to the Christoffel problem, Berg \cite{BergCorpsconvexeset1969} constructed continuous functions $g_d$, $d\ge 2$, on $[-1,1)$ such that a centered measure on $S^{d-1}$ is the first area measure of a convex body if and only if 
 		\begin{align*}
 			\int_{S^{d-1}}g_d(\langle\cdot,v\rangle)d\mu(v)
 		\end{align*}
 		is a support function (where the integral is once again understood as a spherical convolution). These functions are known as Berg's functions, and diverge to $-\infty$ for $t\rightarrow1$ for $d\ge 3$. It was later shown by Goodey and Weil \cite{GoodeyWeilSumssectionssurface2014} that these functions are related to the mean section operators and that this connects mean section bodies to the general Minkowski problem. In our notation, they showed that the $d$th centered mean section operator on $\mathcal{K}(\R^n)$ corresponds to a multiple of $\Phi_{n+1-d}(g_d)$.\\
 		Note that this is actually well defined: It was shown by Berg in \cite[Theorem 3.3]{BergCorpsconvexeset1969} and \cite[p. 31]{BergCorpsconvexeset1969} that
 		\begin{align*}
 			\int_{-1}^1 |g_d(t)|(1-t^2)^{\frac{d-4}{2}}dt&<\infty &&\text{for}~d\ge 3,\\
 			\lim\limits_{t\rightarrow1}(1-t^2)^{\frac{d-3}{2}+\epsilon}g_d(t)&=0 &&\text{for every}~\epsilon>0.
 		\end{align*}
 		Together with the fact that $g_d\in C([-1,1))$, this implies that $g_d,|g_d|\in D^{\frac{d-2}{2}}=D^{\frac{n-(n+1-d)-1}{2}}$ for $d\ge 3$. Since $g_2$ is continuous and $S_0(K)$ is the spherical Lebesgue measure for every $K\in\mathcal{K}(\R^n)$, Proposition \ref{proposition:GeneralIntegrability} implies the following.
 		\begin{corollary}
 			\label{corollary:IntegrabilityBergsFunctions}
 			For every $2\le d\le n+1$, the function $v\mapsto g_d(v_n)$ is integrable with respect to $S_{n+1-d}(K)$ for every $K\in\mathcal{K}(\R^n)$.
 		\end{corollary} 
		
\section{Proof of the characterization result}
\label{section:Classification}
\subsection{Preliminary considerations}
	\begin{lemma}
		\label{lemma:valuesPhionCones}
		For $f\in D^{\frac{n-j-1}{2}}$, $s\in[-1,1]\setminus\{0\}$,
		\begin{align*}
			&\Phi_j(f)\left[C_{\frac{\sqrt{1-s^2}}{s}}\right]=\omega_{n-1} I_{\frac{n-j-1}{2}}(f)[s].
		\end{align*}
	\end{lemma}
	\begin{proof}
		Note that both sides depend continuously on $f\in D^{\frac{n-j-1}{2}}$ for every fixed $s\in [-1,1]\setminus\{0\}$ by Theorem \ref{theorem:ExtensionPhi} and  Corollary \ref{corollary:IaProperties}. Moreover, the equation holds for $f\in C([-1,1])$ by Lemma \ref{lemma:IntegrationCones}, which is a dense subspace of $D^{\frac{n-j-1}{2}}$ by Lemma \ref{lemma:ContFctDenseDj}. Thus it holds for all $f\in D^{\frac{n-j-1}{2}}$.
	\end{proof}
	\begin{corollary}
		\label{corollary:KernelPhij}
		If $f\in D^{\frac{n-j-1}{2}}$ satisfies $\Phi_j(f)=0$, then $f$ is linear.
	\end{corollary}
	\begin{proof}
		Assume that $f\in D^{\frac{n-j-1}{2}}$ satisfies $\Phi_j(f)=0$. By Lemma \ref{lemma:valuesPhionCones}, $I_{\frac{n-j-1}{2}}(f)=0$, and so Lemma \ref{lemma:CompositionJaIa} shows that 
		\begin{align*}
			0=J_{\frac{n-j-1}{2}}\circ I_{\frac{n-j-1}{2}}(f)[s]=f(s)-(n-j-1)s\int_{0}^1\frac{f(t)-f(-t)}{2}(1-t^2)^{\frac{n-j-3}{2}}dt,
		\end{align*}
		so $f$ is linear. 
	\end{proof}

	For $\mu\in\Val_j(\R^n)$, $1\le j\le n-1$, we define the function 
	\begin{align*}
		\phi_\mu:[-1,1]\setminus\{0\}&\rightarrow\R\\
		s&\mapsto \mu\left(C_{\frac{\sqrt{1-s^2}}{s}}\right).
	\end{align*}
	Note that Lemma \ref{lemma:valuesPhionCones} implies that $\phi_\mu\in \mathcal{C}$ if $\mu$ belongs to the image of $\Phi_j$, compare Corollary \ref{corollary:IaProperties}. We will show that this property holds for arbitrary valuations. We require the following observation for truncated cones.
	
	\begin{lemma}
		\label{lemma:relationTruncatedCones}
		Define the following convex bodies for $0<\epsilon<1$, $h\in\R$:
		\begin{align*}
			\tilde{C}_{h,\epsilon}:=&\{x\in C_h: |x_n|\ge  \epsilon|h|\}=(1-\epsilon)C_h+\epsilon h e_n,\\
			D_{h,\epsilon}:=&\{x\in C_h: |x_n|\le  \epsilon|h|\}.
		\end{align*}
		Then 
		\begin{align*}
			&\tilde{C}_{h,\epsilon}\cup D_{h,\epsilon}=C_h, &&\tilde{C}_{h,\epsilon}\cap D_{h,\epsilon}=(1-\epsilon)\D^{n-1}+\epsilon he_n.
		\end{align*}
		Moreover, for $\mu\in\Val_j(\R^n)$:
		\begin{align*}
			\mu(D_{h,\epsilon})=\left[1-(1-\epsilon)^j\right]\mu(C_h)+(1-\epsilon)^j\mu(\D^{n-1}).
		\end{align*}
	\end{lemma}
	\begin{proof}
		The first statement is clear. For the second, we use that $\mu$ is a $j$-homogeneous and translation invariant valuation, so
		\begin{align*}
			\mu(C_h)=&\mu(\tilde{C}_{h,\epsilon})+\mu(D_{h,\epsilon})-\mu(\tilde{C}_{h,\epsilon}\cap D_{h,\epsilon})\\
			=&(1-\epsilon)^j\mu(C_h)+\mu(D_{h,\epsilon})-(1-\epsilon)^j\mu(\D^{n-1}).
		\end{align*}
	\end{proof}

\begin{proposition}
	\label{proposition:EvaluationMapValToC}
	For $\mu\in\Val_j(\R^n)$, $1\le j\le n-1$, the function $\phi_\mu:[-1,1]\setminus\{0\}\rightarrow\R$
	has the following properties:
	\begin{enumerate}
		\item $\phi_\mu$ is continuous on $[-1,1]\setminus\{0\}$.
		\item $\phi_\mu(1)=\phi_\mu(-1)$.
		\item The limit $\lim\limits_{ s\rightarrow0}|s|\phi_\mu(s)$ exists and is finite.
	\end{enumerate}
	In particular, $\phi_\mu\in \mathcal{C}$. Moreover, the map
	\begin{align*}
		\phi:\Val_j(\R^n)&\rightarrow \mathcal{C}\\
		\mu&\mapsto \phi_\mu
	\end{align*}
	is continuous.
\end{proposition}
\begin{proof}
	The first property is trivial, while the second follows from the fact that $\phi_\mu(1)=\mu(C_0)=\phi_\mu(-1)$. For the third, consider for $0<\epsilon<1$, $h\in\R$, the convex bodies defined in Lemma \ref{lemma:relationTruncatedCones}. Since
	\begin{align*}
			\mu(C_h)=&(1-\epsilon)^j\mu(C_h)+\mu(D_{h,\epsilon})-(1-\epsilon)^j\mu(\D^{n-1})\\
			=&\sum_{i=0}^j\binom{j}{i}(-1)^i\epsilon^i \mu(C_h)+\mu(D_{h,\epsilon})-(1-\epsilon)^j\mu(\D^{n-1}),
	\end{align*}
	we obtain by rearranging
	\begin{align*}
		\epsilon\mu(C_h)=\frac{\mu(D_{h,\epsilon})-(1-\epsilon)^j\mu(\D^{n-1})}{\frac{1-(1-\epsilon)^j}{\epsilon}}=\frac{\mu(D_{h,\epsilon})-(1-\epsilon)^j\mu(\D^{n-1})}{-\sum_{i=1}^{j}\binom{j}{i}(-1)^i\epsilon^{i-1}}.
	\end{align*}
	For $h=\frac{\sqrt{1-s^2}}{s}$, $\epsilon=|s|$, we therefore obtain
	\begin{align}
		\label{eq:formulaNormPhiMu}
		|s|\mu\left(C_{\frac{\sqrt{1-s^2}}{s}}\right)=\frac{\mu\left(D_{\frac{\sqrt{1-s^2}}{s},|s|}\right)-(1-|s|)^j\mu(\D^{n-1})}{-\sum_{i=1}^{j}\binom{j}{i}(-1)^i|s|^{i-1}},
	\end{align}
	where 
	\begin{align*}
		D_{\frac{\sqrt{1-s^2}}{s},|s|}=\left\{x\in C_{\frac{\sqrt{1-s^2}}{s}}: |x_n|\le \sqrt{1-s^2}\right\}
	\end{align*}
	converges to $\D^{n-1}\times[0,\pm1]$ for $s\rightarrow0^\pm$. Consequently,
	\begin{align*}
		\lim\limits_{s\rightarrow0^\pm}|s|\mu\left(C_{\frac{\sqrt{1-s^2}}{s}}\right)=\lim\limits_{s\rightarrow0^\pm}\frac{\mu\left(D_{\frac{\sqrt{1-s^2}}{s},|s|}\right)-(1-|s|)^j\mu(\D^{n-1})}{-\sum_{i=1}^{j}\binom{j}{i}(-1)^i|s|^{i-1}}=\frac{\mu(\D^{n-1}\times[0,\pm1])-\mu(\D^{n-1})}{j}.
	\end{align*}
	Thus the two limits $\lim\limits_{s\rightarrow0^\pm}|s|\mu\left(C_{\frac{\sqrt{1-s^2}}{s}}\right)$ exist and coincide because $\mu$ is translation invariant. This completes the proof that $\phi_\mu\in \mathcal{C}$.\\
	
	In order to see that the map $\phi:\Val_j(\R^n)\rightarrow\mathcal{C}$ is continuous, note that \eqref{eq:formulaNormPhiMu} implies that
	\begin{align*}
		\|\phi_\mu\|_{\mathcal{C}}=&\sup_{s\in[-1,1]\setminus\{0\}}|s|\mu\left(C_{\frac{\sqrt{1-s^2}}{s}}\right)=\sup_{s\in[-1,1]\setminus\{0\}} \left|\frac{\mu\left(D_{\frac{\sqrt{1-s^2}}{s},|s|}\right)-(1-|s|)^j\mu(\D^{n-1})}{\frac{1-(1-|s|)^j}{|s|}}\right|.	
	\end{align*}
	Since 
	\begin{align*}
		D_{\frac{\sqrt{1-s^2}}{s},|s|}=\left\{x\in C_{\frac{\sqrt{1-s^2}}{s}}: |x_n|\le \sqrt{1-s^2}\right\}&\subset B_2(0)\quad\text{for all}~s\in [-1,1]\setminus\{0\},
	\end{align*}
	we obtain
	\begin{align*}
		\|\phi_\mu\|_{\mathcal{C}}\le&2^j\|\mu\|\sup_{s\in[-1,1]\setminus\{0\}} \frac{1+(1-|s|)^j}{\frac{1-(1-|s|)^j}{|s|}},
	\end{align*}
	which is bounded since $g(t):= \frac{1-(1-t)^j}{t}$ extends to a continuous function on $[0,1]$ with $g(0)=j$ and consequently has no zeros in $[0,1]$. Since $\mu\mapsto \phi_\mu$ is linear, this map is therefore continuous.
\end{proof}

\begin{corollary}
	\label{corollary:ReplaceValuationOnCones}
	For every $\mu\in \Val_j(\R^n)$ there exists $f\in D^{\frac{n-j-1}{2}}$ such that
	\begin{align*}
		\mu(C_h)=\Phi_j(f)[C_h]\quad\text{for all}~h\in\R.
	\end{align*}
\end{corollary}
\begin{proof}
	Set $f:=\frac{1}{\omega_{n-1}}J_{\frac{n-j-1}{2}}(\phi_\mu)\in D^{\frac{n-j-1}{2}}$, which is well defined due to Proposition \ref{proposition:ContinutiyJa} and Lemma \ref{proposition:EvaluationMapValToC}. This function has the desired property by Proposition \ref{proposition:ContinutiyJa} and Lemma \ref{lemma:valuesPhionCones}.
\end{proof}

\subsection{Proof of Theorem \ref{maintheorem:Classification} by approximation}
	\label{section:ClassificationProof1}
	 Our first proof of Theorem \ref{maintheorem:Classification} relies on the following simple consequence of the classification of smooth spherical valuations by Schuster and Wannerer in Theorem \ref{theorem:SchusterWannererClassification}.
	\begin{lemma}
		\label{lemma:densitiyValContinuousDensity}
		Let $1\le j\le n-1$. Valuations of the form
		\begin{align*}
			\mu(K)=\int_{S^{n-1}}f(v_n)dS_j(K,v)\quad\text{for all}~K\in\mathcal{K}(\R^n)
		\end{align*}
		for $f\in C([-1,1])$ form a dense subspace of $\Val_j(\R^n)^{\SO(n-1)}$.
	\end{lemma}

	\begin{proof}[Proof of Theorem \ref{maintheorem:Classification}]
		The uniqueness statement follows from Corollary \ref{corollary:KernelPhij} and the fact that $\Phi_j$ is linear.\\
		Let $1\le j\le n-2$. Consider the map 
		\begin{align*}
			Z:\Val_j(\R^n)^{\SO(n-1)}&\rightarrow D^{\frac{n-j-1}{2}}\\
			\mu&\mapsto \frac{1}{\omega_{n-1}}J_{\frac{n-j-1}{2}}(\phi_\mu).
		\end{align*}
		Since $\Val_j(\R^n)^{\SO(n-1)}\rightarrow \mathcal{C}$, $\mu\mapsto \phi_\mu$, is well defined and continuous by Proposition \ref{proposition:EvaluationMapValToC} and $J_{\frac{n-j-1}{2}}:\mathcal{C}\rightarrow D^{\frac{n-j-1}{2}}$ is continuous by Proposition \ref{proposition:ContinutiyJa}, $Z$ is well defined and continuous.\\
		We consider the map
		\begin{align*}
			\Val_j(\R^n)^{\SO(n-1)}&\rightarrow \Val_j(\R^n)^{\SO(n-1)}\\
			\mu&\mapsto \mu- \Phi_j\circ Z(\mu).
		\end{align*}
		This map is well defined and continuous. By Lemma \ref{lemma:valuesPhionCones} and Lemma \ref{lemma:densitiyValContinuousDensity}, it vanishes on a dense subspace. Thus is has to vanish identically, so
		\begin{align*}
			\mu=\Phi_j\circ Z(\mu)
		\end{align*}
		for every $\mu\in \Val_{j}(\R^n)^{\SO(n-1)}$, which completes the proof.
	\end{proof}

	\begin{proof}[Proof of Theorem \ref{maintheorem:topologicalIsomorphism}]
		By Corollary \ref{corollary:KernelPhij}, the kernel of $\Phi_j$ is given by linear functions. Consequently $\Phi_j$ descends to a continuous linear map on the quotient, which is therefore injective. Theorem \ref{maintheorem:Classification} implies that this map is also onto. Note that both of these spaces are Banach spaces, compare Lemma \ref{lemma:DaBanach}. Thus $\Phi_j$ defines a continuous and bijective linear map between Banach spaces and is thus a topological isomorphism by the open mapping theorem.
	\end{proof}

\subsection{Proof of Theorem \ref{maintheorem:Classification} using the Hard Lefschetz Theorem}
	\label{section:ClassificationProof2}
	In this section we present a different proof of Theorem \ref{maintheorem:Classification} which relies on the Hard Lefschetz Theorem for spherical valuations, compare Proposition \ref{proposition:LefschetzInjectiveZonal}. In particular, it is implicitly also based on Theorem \ref{theorem:SchusterWannererClassification}. Nevertheless, the proof is of independent interest due to the fact that it provides a different perspective on the role of the cones $C_h$. We refer to the next section for a further application.\\
	
	We will call a convex body $L\in\mathcal{K}(\R^n)$ a body of revolution with axis $e_n$ if $L$ is $\SO(n-1)$-invariant.
	
	\begin{lemma}
		\label{lemma:vanishingConeImpliesBodiesRevolution}
		If $\mu\in \Val_j(\R^n)$ satisfies $\mu(C_h)=0$ for all $h\in\R$, then $\mu$ vanishes on all bodies of revolution with axis $e_n$.
	\end{lemma}
	\begin{proof}
		First, note that $\mu$ vanishes on all translates and rescaled copies of the disk $\D^{n-1}$ or the cones $C_h$ by translation invariance and homogeneity. In particular, $\mu$ vanish on rescaled and translated copies of sets of the form
		\begin{align*}
				D_{h,\epsilon}:=&\{x\in C_h: |x_n|\le  \epsilon|h|\}
		\end{align*}
		by Lemma \ref{lemma:relationTruncatedCones}. Now let $K\in\mathcal{K}(\R^n)$ be a body of revolution with axis $e_n$. Then there exist $a<b$ and a concave function $f:[a,b]\rightarrow\R$ such that
		\begin{align*}
			K=\{(x',x_n)\in\R^n: |x'|\le f(x_n)~\text{for}~x_n\in[a,b]\}.
		\end{align*}
		Let $f_m:[a,b]\rightarrow\R$ be the function that is affine on $\left[a+\frac{i-1}{m}(b-a),a+\frac{i}{m}(b-a)\right]$ with $f_m(a+\frac{i}{m}(b-a))=f(a+\frac{i}{m}(b-a))$ for $i=0,\dots, m$. Then it is easy to see that $f_m$ is concave, so
		\begin{align*}
			K_m=\{(x',x_n)\in\R^n: |x'|\le f_m(x_n)~\text{for}~x_n\in[a,b]\}
		\end{align*}
		is a convex body. Moreover, $(K_m)_m$ converges to $K$ in the Hausdorff metric. By construction, $K_m$ can be written as a union of rescaled translates of the sets $D_{h,\epsilon}$ such that the intersection of two such sets is a translated and rescaled copy of $\D^{n-1}$. By induction on the number of sets in such a decomposition, the valuation property thus implies $\mu(K_m)=0$ for all $m\in\mathbb{N}$ and therefore $\mu(K)=0$ by continuity.
	\end{proof}

	\begin{lemma}
		\label{lemma:vanishingCons1Homogenous}
		If $\mu\in \Val_1(\R^n)^{\SO(n-1)}$ vanishes on all bodies of revolution with axis $e_n$, then $\mu=0$.
	\end{lemma}
	\begin{proof}
		Since $\mu$ is $1$-homogeneous, it is additive, that is, 
		\begin{align*}
			\mu(K+L)=\mu(K)+\mu(L),
		\end{align*}
		compare \cite[Remark 6.3.3]{SchneiderConvexbodiesBrunn2014}. Fix $K\in\mathcal{K}(\R^n)$ and consider the convex body $L\in\mathcal{K}(\R^n)$ defined by
		\begin{align*}
			h_L(y)=\int_{\SO(n-1)}h_K(g^{-1}y)dg,
		\end{align*}
		where we integrate with respect to the unique $\SO(n-1)$-invariant probability measure. Note that $L$ is a body of revolution with axis $e_n$ by construction, so $\mu(L)=0$ by assumption. We may approximate the integral by an appropriate Riemann sum: For $m\in\mathbb{N}$ choose a disjoint decomposition of $\SO(n-1)$ into a finite  cover $(B_{m,i})_{i=1}^{k_m}$ of Borel sets of diameter smaller then $\frac{1}{m}$ and fix $g_{m,i}\in B_{m,i}$. Then
		\begin{align*}
			h_L(y)=\int_{\SO(n-1)}h_L(g^{-1}y)dg=\lim_{m\rightarrow\infty}\sum_{i=1}^{k_m}h_K(g_{m,i}^{-1}y) \int_{B_{m,i}}dg,
		\end{align*}
		and the convergence is uniform in $y\in S^{n-1}$. In particular, the bodies $L_m$ defined by
		\begin{align*}
			h_{L_m}(y)=\sum_{i=1}^{k_m}h_K(g_{m,i}^{-1}y) \int_{B_{m,i}}dg
		\end{align*}
		converge to $L$ in the Hausdorff metric. Note that by the properties of support functions,
		\begin{align*}
			L_m=\sum_{i=1}^{k_m}\int_{B_{m,i}}dg\ g_{m,i}K,
		\end{align*}
		so since $\mu$ is additive and $\SO(n-1)$-invariant, 
		\begin{align*}
			\mu(L_m)=\sum_{i=1}^{k_m}\int_{B_{m,i}}dg\ \mu(g_{m,i}K)=\sum_{i=1}^{k_m}\int_{B_{m,i}}dg\ \mu(K)=\mu(K).
		\end{align*}
		Since $(L_m)_m$ converges to $L$, this implies
		\begin{align*}
			0=\mu(L)=\lim\limits_{m\rightarrow\infty}\mu(L_m)=\mu(K).
		\end{align*}
		Thus $\mu$ vanishes identically.
	\end{proof}

	\begin{corollary}
		\label{corollary:VanishingOnConvesVanishesIdentically}
		If $\mu\in \Val_j(\R^n)^{\SO(n-1)}$  satisfies $\mu(C_h)=0$ for all $h\in \R$, then $\mu=0$.
	\end{corollary}
	\begin{proof}
		By Lemma \ref{lemma:vanishingConeImpliesBodiesRevolution}, $\mu$ vanishes on all bodies of revolution with axis $e_n$. If $L$ is a body of revolution with axis $e_n$, then the same holds for $L+tB_1(0)$ for all $t\ge 0$. Thus $\mu(L+tB_1(0))=0$ for all $t\ge 0$, and in particular,
		\begin{align*}
			\Lambda^{j-1}\mu(L)=\frac{d^{j-1}}{dt^{j-1}}\Big|_0\mu(L+tB_1(0))=0.
		\end{align*}
		Thus $\Lambda^{j-1}\mu$ vanishes on all bodies of revolution with axis $e_n$. Since $\Lambda$ is $\SO(n)$-equivariant and $\mu$ is $\SO(n-1)$-invariant, $\Lambda^{j-1}\mu$ is $\SO(n-1)$-invariant. As $\Lambda^{j-1}\mu$ is $1$-homogeneous and vanishes on all bodies of revolution with axis $e_n$, it has to vanish identically by Lemma \ref{lemma:vanishingCons1Homogenous}. Since $\Lambda^{j-1}:\Val_j(\R^n)^{\SO(n-1)}\rightarrow
		\Val_1(\R^n)^{\SO(n-1)}$ is injective by Proposition \ref{proposition:LefschetzInjectiveZonal}, this shows that $\mu=0$.
	\end{proof}

	\begin{proof}[Second proof of Theorem \ref{maintheorem:Classification}]
		As in the proof in the previous section, the uniqueness statement follow from Corollary \ref{corollary:KernelPhij}.\\
		Let $\mu\in \Val_j(\R^n)^{\SO(n-1)}$. By Corollary \ref{corollary:ReplaceValuationOnCones} there exists a function $f\in D^{\frac{n-j-1}{2}}$ such that
		\begin{align*}
			\mu(C_h)=\Phi_j(f)[C_h]\quad\text{for all}~h\in \R.
		\end{align*}
		Consequently, $\tilde{\mu}:=\mu-\Phi_j(f)\in\Val_j(\R^n)^{\SO(n-1)}$ satisfies $\tilde{\mu}(C_h)=0$ for all $h\in \R$, and thus $\tilde{\mu}=0$ by Corollary 
		\ref{corollary:VanishingOnConvesVanishesIdentically}. Thus, $\mu=\Phi_j(f)$, which completes the proof.
	\end{proof}

\section{From valuations on convex bodies to convex functions}
	\label{section:FunctionalHadwiger}
	In this section we show that Theorem \ref{maintheorem:Classification} can be used to obtain a classification of certain $\SO(n)$-invariant valuations on the space $\Conv(\R^n,\R)$ of all convex functions $f:\R^n\rightarrow\R$. This space is naturally equipped with the topology of uniform convergence on compact subsets (which is equivalent to the topology induced by epi-convergence or pointwise convergence, compare \cite[Theorem
	7.17]{RockafellarWetsVariationalanalysis1998}), and we call a functional $\mu:\Conv(\R^n,\R)\rightarrow\R$ a valuation if 
	\begin{align*}
		\mu(f)+\mu(h)=\mu(\max(f,h))+\mu(\min(f,h))
	\end{align*}
	for all $f,h\in\Conv(\R^n,\R)$ such that the pointwise minimum belongs to $\Conv(\R^n,\R)$. Let $\VConv(\R^n)$ denote the space of all continuous valuations on $\Conv(\R^n,\R)$ that are dually epi-translation invariant, that is, that satisfy
	\begin{align*}
		\mu(f+\ell)=\mu(f)\quad\text{for all}~f\in\Conv(\R^n,\R),~\ell:\R^n\rightarrow\R~\text{affine}.
	\end{align*}
	As shown by Colesanti, Ludwig, and Mussnig \cite{ColesantiEtAlhomogeneousdecompositiontheorem2020}, this space admits a homogeneous decomposition mirroring McMullen's decomposition, 
	\begin{align*}
		\VConv(\R^n)=\bigoplus_{j=0}^n\VConv_j(\R^n),
	\end{align*}
	where $\mu\in \VConv_j(\R^n)$ if and only if $\mu(tf)=t^j\mu(f)$ for all $t\ge 0$ and $f\in\Conv(\R^n,\R)$. Let $\VConv_j(\R^n)^{\SO(n)}$ denote the subspace of all rotation invariant valuations, that is, all $\mu\in\VConv_j(\R^n)$ such that $\mu(f\circ g)=\mu(f)$ for all $g\in \SO(n)$, $f\in\Conv(\R^n,\R)$. The valuations belonging to this space were completely classified by Colesanti, Ludwig, and Mussnig  \cite{ColesantiEtAlHadwigertheoremconvex2020}. In order to state their result, let $C_b((0,\infty))$ denote the space of all continuous functions on $(0,\infty)$ with support bounded from above and consider for $1\le j\le n-1$ the subspace
	\begin{align*}
		D^{n}_j:=\left\{\zeta\in C_b((0,\infty)):\lim\limits_{t\rightarrow0}t^{n-j}\zeta(t)=0,~\lim\limits_{t\rightarrow0}\int_{t}^\infty\zeta(s)s^{n-j-1}ds~\text{exists and is finite}\right\}.
	\end{align*}
	Let us remark that this space is related to the semi-norm $\|\cdot\|_{n-j}$ from Section \ref{section:SingularValuationsEstimates}, compare \cite{KnoerrSingularvaluationsHadwiger2022}.
	Let us state the classification result by Colesanti, Ludwig, and Mussnig in the following form.
	\begin{theorem}
		\label{theorem:hadwigerVConv}
		Let $1\le j\le n-1$. 
		\begin{enumerate}
			\item[a)] For every $\zeta\in D^n_j$ and every $f\in\Conv(\R^n,\R)$ the limit
				\begin{align*}
					V^*_{j,\zeta}(f):=\lim\limits_{\epsilon\rightarrow0}\int_{\R^n\setminus B_\epsilon(0)}\zeta(|x|)d\Hess_j(f,x)
				\end{align*}
				exists, and this defines a continuous valuation $V^*_{j,\zeta}\in \VConv_j(\R^n)^{\SO(n)}$.
			\item[b)] For every $\mu\in \VConv_j(\R^n)^{\SO(n)}$ there exists a unique $\zeta\in D^n_j$ such that $\mu= V^*_{j,\zeta}$.
		\end{enumerate}
	\end{theorem}
	Here, $\Hess_j(f)$ denotes the $j$th Hessian measure of $f\in\Conv(\R^n,\R)$, compare \cite{ColesantiEtAlHadwigertheoremconvex2020}. Note that we have excluded $j=0$ and $j=n$, since the result takes a much simpler form in these cases. Due to the analogy of the result with Hadwiger's description of $\SO(n)$-invariant valuations in $\Val(\R^n)$ \cite{HadwigerVorlesungenuberInhalt1957}, the functionals $V^*_{i,\zeta}$ are also called \emph{functional intrinsic volumes}. Let us also remark that the representation in a) was established in \cite{KnoerrSingularvaluationsHadwiger2022}. Different representation formulas where obtained by Colesanti, Ludwig, and Mussnig in \cite{ColesantiEtAlHadwigertheoremconvex,ColesantiEtAlHadwigertheoremconvex2022}. Further proofs of Theorem \ref{theorem:hadwigerVConv} b) can be found in \cite{ColesantiEtAlHadwigertheoremconvex2023,KnoerrSingularvaluationsHadwiger2022}. \\
	
	The goal of this section is a proof of Theorem \ref{theorem:hadwigerVConv} b) using the results in Section \ref{section:ClassificationProof2}. We consider the map $\mathcal{R}^{n-j}:D^n_j\rightarrow C_c([0,\infty))$ given by
	\begin{align*}
		\mathcal{R}^{n-j}(\zeta)[t]=t^{n-j}\zeta(j)+(n-j)\int_t^\infty\zeta(s)s^{n-j-1}ds.
	\end{align*}
	By \cite[Lemma 3.7]{ColesantiEtAlHadwigertheoremconvex}, $\mathcal{R}^{n-j}$ is a well defined bijection.
	Consider the function $u_t\in \Conv(\R^n,\R)$ given for $t\ge0$ by $u_t(x)=\max(0,|x|-t)$ (and note that these functions are intimately related to the support functions of the cones $C_h$, compare \eqref{eq:relationUtCone} below).
	\begin{lemma}[\cite{ColesantiEtAlHadwigertheoremconvex2020} Lemma 2.15]
		\label{lemma:values_ut}
		For $\zeta\in D^n_j$, $t\ge0$: $V^*_{j,\zeta}(u_t)=\omega_n\binom{n}{j}\mathcal{R}^{n-j}(\zeta)[t]$.
	\end{lemma}
	Let $\mu\in\VConv_j(\R^n)$ be a given valuation. Since $t\mapsto u_t$ is a continuous curve in $\Conv(\R^n,\R)$ and $\mu$ is continuous, the function $\phi(t)=\mu(u_t)$ is continuous on $[0,\infty)$. 
	It was shown in \cite{Knoerrsupportduallyepi2021} that any $\mu\in\VConv_j(\R^n)$ is compactly supported in the following sense: There exists a compact subset $A$ of $\R^n$ with the property that whenever $f,h\in\Conv(\R^n,\R)$ coincide on a neighborhood of $A$, then $\mu(f)=\mu(h)$. In particular, for any $\mu\in \VConv_j(\R^n)$ there exists $R>0$ such that $u_t$ coincides with the zero function on a neighborhood of such a set for $t>R$, so $\mu(u_t)=\mu(0)$ for all $t> R$. If $j\ge 1$, $\mu(0)=0$, so $\phi\in C_c([0,\infty))$ has compact support.

	\begin{proof}[New proof of Theorem \ref{theorem:hadwigerVConv} b)]
		Let $1\le j\le n-1$ and $\mu\in\VConv_j(\R^n)^{\SO(n)}$. Consider the function $\phi(t):=\mu(u_t)$, which belongs to $C_c([0,\infty))$ by the previous discussion. Define $\zeta\in D^n_j$ by
		\begin{align*}
			\zeta:=\frac{1}{\omega_n\binom{n}{j}}\left(\mathcal{R}^{n-j}\right)^{-1}(\phi),
		\end{align*}
		which is well defined by \cite[Lemma 3.7]{ColesantiEtAlHadwigertheoremconvex}. Theorem \ref{theorem:hadwigerVConv} a) and Lemma \ref{lemma:values_ut} imply that $\mu(u_t)=V^*_{j,\zeta}(u_t)$ for all $t\in[0,\infty)$. Set $\tilde{\mu}:=\mu-V^*_{j,\zeta}$, so that $\tilde{\mu}(u_t)=0$ for all $t\in[0,\infty)$. We will show that $\tilde{\mu}=0$.\\
		
		Define $T(\tilde{\mu}):\mathcal{K}(\R^n\times\R)\rightarrow\R$ by
		\begin{align*}
			T(\tilde{\mu})[K]:=\tilde{\mu}(h_K(\cdot,-1)).
		\end{align*}
		By \cite[Theorem 3.3]{Knoerrsupportduallyepi2021}, $T(\tilde{\mu})\in\Val_j(\R^{n}\times\R)$ and $\tilde{\mu}=0$ if and only if $T(\tilde{\mu})=0$. By the basic properties of support functions, $T(\tilde{\mu})$ is $\SO(n)$-invariant, where we identify $\SO(n)$ with the subgroup of $\SO(n+1)$ that leaves the axis spanned by $e_{n+1}$ invariant, so $T(\tilde{\mu})\in \Val_j(\R^n\times\R)^{\SO(n)}$. In particular, Corollary \ref{corollary:VanishingOnConvesVanishesIdentically} shows that $T(\tilde{\mu})$ vanishes identically if it vanishes on the cones $C_h\subset \R^{n+1}$ for $h\in\R$. However, for $h\in\R$
		\begin{align}
			\label{eq:relationUtCone}
			\begin{split}
				h_{C_h}(\cdot,-1)=\max(h_{\{he_{n+1}\}}(\cdot,-1),h_{\D^n}(\cdot,-1))=\begin{cases}
					u_{-h}-h,  & h\le 0,\\
					|\cdot|, & h>0.
				\end{cases}
			\end{split}
		\end{align}
		In the second case, this equals $u_0$. Since $\tilde{\mu}$ is dually epi-translation invariant, we thus obtain
		\begin{align*}
			T(\tilde{\mu})[C_h]=\tilde{\mu}(h_{C_h}(\cdot,-1))=\begin{cases}
				\tilde{\mu}(u_{-h}), & h\le 0,\\
				\tilde{\mu}(u_0), & h>0,
			\end{cases}
		\end{align*}
		which vanishes in both cases. Corollary \ref{corollary:VanishingOnConvesVanishesIdentically} implies $T(\tilde{\mu})=0$ and thus $\tilde{\mu}=0$. This shows $\mu=V^*_{j,\zeta}$, which completes the proof.
	\end{proof}
	\begin{remark}
		Note that the proof given above uses Corollary \ref{corollary:VanishingOnConvesVanishesIdentically} to show that a valuation in $\VConv_j(\R^n)^{\SO(n)}$ vanishes identically as soon as it vanishes on the functions $u_t$, $t\ge 0$, and deduces Theorem \ref{theorem:hadwigerVConv} b) from this result. This approach was already pointed out by Colesanti, Ludwig, and Mussnig as a possible consequence of proof of \cite[Lemma 8.3]{ColesantiEtAlHadwigertheoremconvex2022} that does not rely on Theorem \ref{theorem:hadwigerVConv} b). In particular, the argument given above can be used to obtain this independent proof.
	\end{remark}
	
\bibliography{../../../library/library.bib}

\begin{thebibliography}{10}

\bibitem{AbardiaBernigProjectionbodiescomplex2011}
Judit Abardia and Andreas Bernig.
\newblock Projection bodies in complex vector spaces.
\newblock {\em Adv. Math.}, 227(2):830--846, 2011.

\bibitem{AleskerDescriptioncontinuousisometry1999}
Semyon Alesker.
\newblock Description of continuous isometry covariant valuations on convex
  sets.
\newblock {\em Geom. Dedicata}, 74(3):241--248, 1999.

\bibitem{AleskerDescriptiontranslationinvariant2001}
Semyon Alesker.
\newblock Description of translation invariant valuations on convex sets with
  solution of {P}. {M}c{M}ullen's conjecture.
\newblock {\em Geom. Funct. Anal.}, 11(2):244--272, 2001.

\bibitem{AleskerHardLefschetztheorem2003}
Semyon Alesker.
\newblock Hard {L}efschetz theorem for valuations, complex integral geometry,
  and unitarily invariant valuations.
\newblock {\em J. Differential Geom.}, 63(1):63--95, 2003.

\bibitem{AleskerTheoryvaluationsmanifolds.2006}
Semyon Alesker.
\newblock Theory of valuations on manifolds. {I}. {L}inear spaces.
\newblock {\em Israel J. Math.}, 156:311--339, 2006.

\bibitem{AleskerIntroductiontheoryvaluations2018}
Semyon Alesker.
\newblock {\em Introduction to the theory of valuations}, volume 126 of {\em
  CBMS Regional Conference Series in Mathematics}.
\newblock Conference Board of the Mathematical Sciences, Washington, DC; by the
  American Mathematical Society, Providence, RI, 2018.

\bibitem{AleskerEtAlHarmonicanalysistranslation2011}
Semyon Alesker, Andreas Bernig, and Franz~E. Schuster.
\newblock Harmonic analysis of translation invariant valuations.
\newblock {\em Geom. Funct. Anal.}, 21(4):751--773, 2011.

\bibitem{AleskerFaifmanConvexvaluationsinvariant2014}
Semyon Alesker and Dmitry Faifman.
\newblock Convex valuations invariant under the {L}orentz group.
\newblock {\em J. Differential Geom.}, 98(2):183--236, 2014.

\bibitem{BergEtAlLogconcavityproperties2018}
Astrid Berg, Lukas Parapatits, Franz~E. Schuster, and Manuel Weberndorfer.
\newblock Log-concavity properties of {M}inkowski valuations.
\newblock {\em Trans. Amer. Math. Soc.}, 370(7):5245--5277, 2018.
\newblock With an appendix by Semyon Alesker.

\bibitem{BergCorpsconvexeset1969}
Christian Berg.
\newblock Corps convexes et potentiels sph\'eriques.
\newblock {\em Mat.-Fys. Medd. Danske Vid. Selsk.}, 37(6):64, 1969.

\bibitem{BernigAlgebraicintegralgeometry2011}
Andreas Bernig.
\newblock Algebraic integral geometry.
\newblock In {\em Global differential geometry}, volume~17 of {\em Springer
  Proc. Math.}, pages 107--145. Springer, Heidelberg, 2012.

\bibitem{BernigBroeckerValuationsmanifoldsRumin2007}
Andreas Bernig and Ludwig Br\"{o}cker.
\newblock Valuations on manifolds and {R}umin cohomology.
\newblock {\em J. Differential Geom.}, 75(3):433--457, 2007.

\bibitem{BernigFaifmanGeneralizedtranslationinvariant2016}
Andreas Bernig and Dmitry Faifman.
\newblock Generalized translation invariant valuations and the polytope
  algebra.
\newblock {\em Adv. Math.}, 290:36--72, 2016.

\bibitem{BernigFuHermitianintegralgeometry2011}
Andreas Bernig and Joseph H.~G. Fu.
\newblock Hermitian integral geometry.
\newblock {\em Ann. of Math. (2)}, 173(2):907--945, 2011.

\bibitem{BernigHugKinematicformulastensor2018}
Andreas Bernig and Daniel Hug.
\newblock Kinematic formulas for tensor valuations.
\newblock {\em J. Reine Angew. Math.}, 736:141--191, 2018.

\bibitem{BraunerOrtegaMorenoFixedPointsMean2023}
Leo Brauner and Oscar Ortega-Moreno.
\newblock Fixed points of mean section operators.
\newblock {\em arXiv:2302.11973}, 2023.

\bibitem{ColesantiEtAlHadwigertheoremconvex}
Andrea Colesanti, Monika Ludwig, and Fabian Mussnig.
\newblock The {H}adwiger theorem on convex functions, {II}: {C}auchy-{K}ubota
  formulas.
\newblock {\em Amer. J. Math., in press.}

\bibitem{ColesantiEtAlHadwigertheoremconvex2020}
Andrea Colesanti, Monika Ludwig, and Fabian Mussnig.
\newblock The {H}adwiger theorem on convex functions. {I}.
\newblock {\em arXiv:2009.03702}, 2020.

\bibitem{ColesantiEtAlhomogeneousdecompositiontheorem2020}
Andrea Colesanti, Monika Ludwig, and Fabian Mussnig.
\newblock A homogeneous decomposition theorem for valuations on convex
  functions.
\newblock {\em J. Funct. Anal.}, 279(5):Paper No. 108573, 25, 2020.

\bibitem{ColesantiEtAlHadwigertheoremconvex2022}
Andrea Colesanti, Monika Ludwig, and Fabian Mussnig.
\newblock The {H}adwiger theorem on convex functions, {III}: {S}teiner formulas
  and mixed {M}onge-{A}mp\`ere measures.
\newblock {\em Calc. Var. Partial Differential Equations}, 61(5):Paper No. 181,
  37, 2022.

\bibitem{ColesantiEtAlHadwigertheoremconvex2023}
Andrea Colesanti, Monika Ludwig, and Fabian Mussnig.
\newblock The {H}adwiger theorem on convex functions, {IV}: {T}he {K}lain
  approach.
\newblock {\em Adv. Math.}, 413:Paper No. 108832, 2023.

\bibitem{DorrekMinkowskiendomorphisms2017}
Felix Dorrek.
\newblock Minkowski endomorphisms.
\newblock {\em Geom. Funct. Anal.}, 27(3):466--488, 2017.

\bibitem{FireyLocalbehaviourarea1970}
William~J. Firey.
\newblock Local behaviour of area functions of convex bodies.
\newblock {\em Pacific J. Math.}, 35:345--357, 1970.

\bibitem{FuMongeAmperefunctions.1989}
Joseph H.~G. Fu.
\newblock Monge-{A}mp\`ere functions. {I}, {II}.
\newblock {\em Indiana Univ. Math. J.}, 38(3):745--771, 1989.

\bibitem{FuStructureunitaryvaluation2006}
Joseph H.~G. Fu.
\newblock Structure of the unitary valuation algebra.
\newblock {\em J. Differential Geom.}, 72(3):509--533, 2006.

\bibitem{FuAlgebraicintegralgeometry2014}
Joseph H.~G. Fu.
\newblock Algebraic integral geometry.
\newblock In {\em Integral geometry and valuations}, Adv. Courses Math. CRM
  Barcelona, pages 47--112. Birkh\"{a}user/Springer, Basel, 2014.

\bibitem{GardnerGeometrictomography2006}
Richard~J. Gardner.
\newblock {\em Geometric tomography}, volume~58 of {\em Encyclopedia of
  Mathematics and its Applications}.
\newblock Cambridge University Press, New York, second edition, 2006.

\bibitem{GoodeyWeilSumssectionssurface2014}
Paul Goodey and Wolfgang Weil.
\newblock Sums of sections, surface area measures, and the general {M}inkowski
  problem.
\newblock {\em J. Differential Geom.}, 97(3):477--514, 2014.

\bibitem{HaberlMinkowskivaluationsintertwining2012}
Christoph Haberl.
\newblock Minkowski valuations intertwining with the special linear group.
\newblock {\em J. Eur. Math. Soc. (JEMS)}, 14(5):1565--1597, 2012.

\bibitem{HaberlSchusterGeneral$L_p$affine2009}
Christoph Haberl and Franz~E. Schuster.
\newblock General {$L_p$} affine isoperimetric inequalities.
\newblock {\em J. Differential Geom.}, 83(1):1--26, 2009.

\bibitem{HaberlSchusterAffinevs.Euclidean2019}
Christoph Haberl and Franz~E. Schuster.
\newblock Affine vs. {E}uclidean isoperimetric inequalities.
\newblock {\em Adv. Math.}, 356:106811, 26, 2019.

\bibitem{HadwigerVorlesungenuberInhalt1957}
Hugo Hadwiger.
\newblock {\em Vorlesungen \"{u}ber {I}nhalt, {O}berfl\"{a}che und
  {I}soperimetrie}.
\newblock Springer, Berlin-G\"{o}ttingen-Heidelberg, 1957.

\bibitem{HugEtAlAdditivekinematicformulas2024}
Daniel Hug, Fabian Mussnig, and Jacopo Ulivelli.
\newblock Additive kinematic formulas for convex functions.
\newblock {\em arXiv:2403.06697}, 2024.

\bibitem{HugEtAlKubotatypeformulas2024}
Daniel Hug, Fabian Mussnig, and Jacopo Ulivelli.
\newblock Kubota-type formulas and supports of mixed measures.
\newblock {\em arXiv:2401.16371}, 2024.

\bibitem{KiderlenBlaschkeMinkowskiendomorphisms2006}
Markus Kiderlen.
\newblock Blaschke- and {M}inkowski-endomorphisms of convex bodies.
\newblock {\em Trans. Amer. Math. Soc.}, 358(12):5539--5564, 2006.

\bibitem{Knoerrsupportduallyepi2021}
Jonas Knoerr.
\newblock The support of dually epi-translation invariant valuations on convex
  functions.
\newblock {\em J. Funct. Anal.}, 281(5):Paper No. 109059, 52, 2021.

\bibitem{KnoerrSingularvaluationsHadwiger2022}
Jonas Knoerr.
\newblock Singular valuations and the {H}adwiger theorem on convex functions.
\newblock {\em arXiv:2209.05158}, 2022.

\bibitem{Knoerrgeometricdecompositionunitarily2024}
Jonas Knoerr.
\newblock A geometric decomposition for unitarily invariant valuations on
  convex functions.
\newblock {\em arXiv:2408.01352}, 2024.

\bibitem{KnoerrMongeAmpereoperators2024}
Jonas Knoerr.
\newblock Monge-{A}mp\`ere operators and valuations.
\newblock {\em Calc. Var. Partial Differential Equations}, 63(4):Paper No. 89,
  34, 2024.

\bibitem{KnoerrSmoothvaluationsconvex2024}
Jonas Knoerr.
\newblock Smooth valuations on convex functions.
\newblock {\em J. Differential Geom.}, 126(2):801--835, 2024.

\bibitem{KnoerrUlivellivaluationsconvexbodies2024}
Jonas Knoerr and Jacopo Ulivelli.
\newblock From valuations on convex bodies to convex functions.
\newblock {\em Math. Ann.}, 2024.

\bibitem{LudwigProjectionbodiesvaluations2002}
Monika Ludwig.
\newblock Projection bodies and valuations.
\newblock {\em Adv. Math.}, 172(2):158--168, 2002.

\bibitem{LudwigEllipsoidsmatrixvalued2003}
Monika Ludwig.
\newblock Ellipsoids and matrix-valued valuations.
\newblock {\em Duke Math. J.}, 119(1):159--188, 2003.

\bibitem{LudwigMinkowskivaluations2005}
Monika Ludwig.
\newblock Minkowski valuations.
\newblock {\em Trans. Amer. Math. Soc.}, 357(10):4191--4213, 2005.

\bibitem{LudwigReitznerclassification$SLn$invariant2010}
Monika Ludwig and Matthias Reitzner.
\newblock A classification of {$\mathrm{SL}(n)$} invariant valuations.
\newblock {\em Ann. of Math. (2)}, 172(2):1219--1267, 2010.

\bibitem{LutwakInequalitiesmixedprojection1993}
Erwin Lutwak.
\newblock Inequalities for mixed projection bodies.
\newblock {\em Trans. Amer. Math. Soc.}, 339(2):901--916, 1993.

\bibitem{LutwakEtAl$L_p$affineisoperimetric2000}
Erwin Lutwak, Deane Yang, and Gaoyong Zhang.
\newblock {$L_p$} affine isoperimetric inequalities.
\newblock {\em J. Differential Geom.}, 56(1):111--132, 2000.

\bibitem{McMullenValuationsEulertype1977}
Peter McMullen.
\newblock Valuations and {E}uler-type relations on certain classes of convex
  polytopes.
\newblock {\em Proc. London Math. Soc. (3)}, 35(1):113--135, 1977.

\bibitem{McMullenContinuoustranslationinvariant1980}
Peter McMullen.
\newblock Continuous translation-invariant valuations on the space of compact
  convex sets.
\newblock {\em Arch. Math. (Basel)}, 34(4):377--384, 1980.

\bibitem{OrtegaMorenoIterationsMinkowskivaluations2023}
Oscar Ortega-Moreno.
\newblock Iterations of {M}inkowski valuations.
\newblock {\em J. Funct. Anal.}, 284(10):Paper No. 109887, 30, 2023.

\bibitem{OrtegaMorenoSchusterFixedpointsMinkowski2021}
Oscar Ortega-Moreno and Franz~E. Schuster.
\newblock Fixed points of {M}inkowski valuations.
\newblock {\em Adv. Math.}, 392:Paper No. 108017, 33, 2021.

\bibitem{ParapatitsSchusterSteinerformulaMinkowski2012}
Lukas Parapatits and Franz~E. Schuster.
\newblock The {S}teiner formula for {M}inkowski valuations.
\newblock {\em Adv. Math.}, 230(3):978--994, 2012.

\bibitem{PettyIsoperimetricproblems1971}
C.~M. Petty.
\newblock Isoperimetric problems.
\newblock In {\em Proceedings of the {C}onference on {C}onvexity and
  {C}ombinatorial {G}eometry ({U}niv. {O}klahoma, {N}orman, {O}kla., 1971)},
  pages 26--41. University of Oklahoma, Department of Mathematics, Norman, OK,
  1971.

\bibitem{RockafellarWetsVariationalanalysis1998}
R.~Tyrrell Rockafellar and Roger J.-B. Wets.
\newblock {\em Variational analysis}, volume 317 of {\em Grundlehren der
  mathematischen Wissenschaften [Fundamental Principles of Mathematical
  Sciences]}.
\newblock Springer, Berlin, 1998.

\bibitem{SchneiderEquivariantendomorphismsspace1974}
Rolf Schneider.
\newblock Equivariant endomorphisms of the space of convex bodies.
\newblock {\em Trans. Amer. Math. Soc.}, 194:53--78, 1974.

\bibitem{SchneiderConvexbodiesBrunn2014}
Rolf Schneider.
\newblock {\em Convex bodies: the {B}runn-{M}inkowski theory}, volume 151 of
  {\em Encyclopedia of Mathematics and its Applications}.
\newblock Cambridge University Press, Cambridge, expanded edition, 2014.

\bibitem{SchusterWannererMinkowskivaluationsgeneralized2018}
Franz Schuster and Thomas Wannerer.
\newblock Minkowski valuations and generalized valuations.
\newblock {\em J. Eur. Math. Soc. (JEMS)}, 20(8):1851--1884, 2018.

\bibitem{SchusterCroftonmeasuresMinkowski2010}
Franz~E. Schuster.
\newblock Crofton measures and {M}inkowski valuations.
\newblock {\em Duke Math. J.}, 154(1):1--30, 2010.

\bibitem{SchusterWannerer$GLn$contravariantMinkowski2012}
Franz~E. Schuster and Thomas Wannerer.
\newblock {$\mathrm{GL}(n)$} contravariant {M}inkowski valuations.
\newblock {\em Trans. Amer. Math. Soc.}, 364(2):815--826, 2012.

\bibitem{WangaffineSobolevZhang2012}
Tuo Wang.
\newblock The affine {S}obolev-{Z}hang inequality on {$BV(\Bbb R^n)$}.
\newblock {\em Adv. Math.}, 230(4-6):2457--2473, 2012.

\bibitem{WannererIntegralgeometryunitary2014}
Thomas Wannerer.
\newblock Integral geometry of unitary area measures.
\newblock {\em Adv. Math.}, 263:1--44, 2014.

\bibitem{Wannerermoduleunitarilyinvariant2014}
Thomas Wannerer.
\newblock The module of unitarily invariant area measures.
\newblock {\em J. Differential Geom.}, 96(1):141--182, 2014.

\bibitem{ZhangRestrictedchordprojection1991}
Gao~Yong Zhang.
\newblock Restricted chord projection and affine inequalities.
\newblock {\em Geom. Dedicata}, 39(2):213--222, 1991.

\bibitem{ZhangaffineSobolevinequality1999}
Gaoyong Zhang.
\newblock The affine {S}obolev inequality.
\newblock {\em J. Differential Geom.}, 53(1):183--202, 1999.

\end{thebibliography}
\Addresses
\end{document}